\documentclass[12pt,letterpaper]{scrartcl}

\usepackage{amsmath}
\usepackage{amsthm}
\usepackage{amssymb}
\usepackage{stmaryrd}
\usepackage{accents}
\usepackage{bbm}
\usepackage[mathscr]{euscript}
\usepackage{xcolor}
\definecolor{Myblue}{rgb}{0,0,0.6}  
\usepackage[colorlinks,citecolor=Myblue,linkcolor=Myblue,urlcolor=Myblue,pdfpagemode=None]{hyperref}
\usepackage{subcaption}
\usepackage{graphicx}
\usepackage{epstopdf}
\usepackage[totalwidth=420pt, totalheight=590pt]{geometry}
\usepackage[inline,shortlabels]{enumitem}
\usepackage[nottoc]{tocbibind}
\usepackage{hyperref}
\usepackage{fancyhdr}
\usepackage{xcolor}
\usepackage{tikz}
\usetikzlibrary{decorations.markings}
\usetikzlibrary{knots}
\usetikzlibrary{arrows}
\usetikzlibrary{cd}

\allowdisplaybreaks

\theoremstyle{definition}
\newtheorem{defn}{Definition}
\newtheorem{thm}[defn]{Theorem}
\newtheorem{prp}[defn]{Proposition}
\newtheorem{lem}[defn]{Lemma}
\newtheorem{cor}[defn]{Corollary}

\newtheorem{rem}[defn]{Remark}
\newtheorem{example}[defn]{Example}

\numberwithin{equation}{section}
\numberwithin{defn}{section}
\numberwithin{figure}{section}

\newcommand{\pic}[2][0.75]{
	\begin{tikzpicture}[scale=0.5,baseline={([yshift=-.5ex]current bounding box.center)}]
	\node at (0,0) {\includegraphics[scale=#1]{figures/#2}};
	\end{tikzpicture}
}

\begin{document}
\def\it{\textit}
\def\mcA{\mathcal{A}}
\def\mcB{\mathcal{B}}
\def\mcC{\mathcal{C}}
\def\mcD{\mathcal{D}}
\def\mcI{\mathcal{I}}
\def\mcS{\mathcal{S}}
\def\mcZ{\mathcal{Z}}
\def\mcACA{{_A \mathcal{C}_A}}
\def\mcD{\mathcal{D}}
\def\euT{\mathscr{T}}
\def\opk{\Bbbk}
\def\opA{\mathbb{A}}
\def\opZ{\mathbb{Z}}
\def\opid{\mathbbm{1}}
\def\a{\alpha}
\def\b{\beta}
\def\g{\gamma}
\def\d{\delta}
\def\D{\Delta}
\def\vareps{\varepsilon}
\def\l{\lambda}
\def\abar{\overline{\a}}
\def\bbar{\overline{\b}}
\def\gbar{\overline{\g}}
\def\ddbar{\overline{\d}}
\def\mubar{\overline{\mu}}
\def\pibar{{\bar{\pi}}}

\def\opp{{\operatorname{op}}}
\def\id{\operatorname{id}}
\def\im{\operatorname{im}}
\def\Hom{\operatorname{Hom}}
\def\End{\operatorname{End}}
\def\tr{\operatorname{tr}}
\def\ev{\operatorname{ev}}
\def\coev{\operatorname{coev}}
\def\evt{\widetilde{\operatorname{ev}}}
\def\coevt{\widetilde{\operatorname{coev}}}
\def\Id{\operatorname{Id}}
\def\Vect{\operatorname{Vect}}
\def\loc{{\operatorname{loc}}}
\def\orb{{\operatorname{orb}}}
\def\coker{{\operatorname{coker}}}
\def\Ind{{\operatorname{Ind}}}
\def\Irr{{\operatorname{Irr}}}
\def\Dim{\operatorname{Dim}}

\def\Lra{\Leftrightarrow}
\def\Ra{\Rightarrow}
\def\ra{\rightarrow}
\def\lra{\leftrightarrow}
\def\xra{\xrightarrow}

\newcommand{\eqrefO}[1]{\hyperref[eq:O#1]{\text{(O#1)}}}
\newcommand{\eqrefT}[1]{\hyperref[eq:T#1]{\text{(T#1)}}}

\title{Constructing modular categories from orbifold data}

\author{
Vincentas Mulevi\v{c}ius \quad
Ingo Runkel\\[0.5cm]
\normalsize{\texttt{\href{mailto:vincentas.mulevicius@uni-hamburg.de}{vincentas.mulevicius@uni-hamburg.de}}} \quad
\normalsize{\texttt{\href{mailto:ingo.runkel@uni-hamburg.de}{ingo.runkel@uni-hamburg.de}}}\\[0.1cm]
{\normalsize\slshape Fachbereich Mathematik, Universit\"{a}t Hamburg, Germany}\\[-0.1cm]
}

\date{}
\maketitle

\begin{abstract}
In Carqueville et al., \normalsize{\texttt{\href{https://arxiv.org/abs/1809.01483}{arXiv:1809.01483}}}, the notion of an orbifold datum $\opA$ in a modular fusion category $\mcC$ was introduced as part of a generalised orbifold construction for Reshetikhin-Turaev TQFTs.
In this paper, given a simple orbifold datum $\opA$ in $\mcC$, we introduce a ribbon category $\mcC_\opA$ and show that it is again a modular fusion category.
The definition of $\mcC_\opA$ is motivated by properties of Wilson lines in the generalised orbifold.
We analyse two examples in detail:
(i)  when $\opA$ is given by a simple commutative $\D$-separable Frobenius algebra $A$ in $\mcC$;
(ii) when $\opA$ is an orbifold datum in $\mcC = \Vect$, built from a spherical fusion category $\mcS$.
We show that in case (i), $\mcC_\opA$ is ribbon-equivalent to the category of local modules of $A$, and in case (ii), to the Drinfeld centre of $\mcS$.
The category $\mcC_\opA$ thus unifies these two constructions into a single algebraic setting.
\end{abstract}
\newpage

\setcounter{tocdepth}{2}

\tableofcontents

\section{Introduction}

A modular fusion category (MFC) is a semisimple modular tensor category, that is, a fusion category which is equipped with a braiding and a ribbon twist, such that the braiding satisfies a non-degeneracy condition. Modular fusion categories are important ingredients in several constructions in mathematics and mathematical physics, such as link and three-manifold invariants, two-dimensional conformal field theory, and topological phases of matter.

Three ``intrinsic'' ways to produce examples of MFCs are:
\begin{enumerate}
\item
If $\mathcal{S}$ is a spherical fusion category, one can construct its Drinfeld centre $\mathcal{Z}(\mathcal{S})$, which is a modular fusion category (see \cite[Thm.\ 1.2]{Mu}).
\item 
If $\mathcal{C}$ is a MFC and if $A \in \mathcal{C}$ is a commutative simple special Frobenius algebra, then $\mathcal{C}_A^\mathrm{loc}$, the category of so-called local $A$-modules, is again a MFC (see \cite[Thm.\ 4.5]{KO}).
\item
If $\mathcal{B} = \bigoplus_{g \in G} \mathcal{B}_g$ is a $G$-crossed ribbon category for some finite group $G$ such that its neutral component $\mathcal{B}_e$ is modular, then $\mathcal{B}^G$, its $G$-equivariantisation, aka.\ the gauging by $G$, is a MFC (see \cite[Thm.\ 10.4]{Ki}, \cite[Prop. 4.56]{DGNO}).
\end{enumerate}
In \cite{CRS3} it was found that all three cases produce examples of ``orbifold data'' in a MFC $\mathcal{C}$, which means that they define so-called generalised orbifolds of Reshetikhin-Turaev 3d TQFTs.

Below, we briefly recall the definition of an orbifold datum $\mathbb{A}$ in $\mathcal{C}$ and summarise how it produces a new MFC $\mathcal{C}_\mathbb{A}$. This is the main result of this paper. We then motivate the algebraic construction of $\mathcal{C}_\mathbb{A}$ from 3d TQFT and from the generalised orbifold construction.

\subsection{Algebraic results}

Let $\mathcal{C}$ be a MFC over an algebraically closed field $\Bbbk$. An {\em orbifold datum} in $\mathcal{C}$ is a tuple $\mathbb{A} = (A,T,\alpha,\bar\alpha,\psi,\phi)$, where
\begin{itemize}
	\item[$A$:] a $\Delta$-separable symmetric Frobenius algebra in $\mathcal{C}$ ($\Delta$-separable means that $\mu \circ \Delta=\id_A$, where $\mu$ is the product and $\Delta$ the coproduct of $A$),
	\item[$T$:] an $A$-$AA$-bimodule in $\mathcal{C}$, where we abbreviate $A \otimes A$ by $AA$,
	\item[$\alpha$:] an $A$-$AAA$ bimodule map $T \otimes_{A,2} T \to T \otimes_{A,1} T$, where the additional index indicates over which tensor factor of $A\otimes A$ the tensor product is taken,
	\item[$\bar\alpha$:] up to normalisation the inverse of $\alpha$,
	\item[$\psi,\phi$:] normalisation factors $\psi \in \mathrm{End}_{AA}(A)^\times$ and $\phi \in \Bbbk^\times$.
\end{itemize}
These are subject to further conditions \cite{CRS1,CRS3} which we recall in Section \ref{subsec:orb_data}.
The three examples above give rise to orbifold data as follows \cite{CRS3}.
\begin{enumerate}
\item
\label{eg1}
In this case, $\mathcal{C}=\mathrm{Vect}$, the category of finite-dimensional $\Bbbk$-vector spaces. 
Let $\mathcal{I}$ be a set of representatives of isomorphism classes of simple objects in $\mathcal{S}$. Then
$A = \bigoplus_{i \in \mathcal{I}} \mathrm{End}_\mathcal{S}(i)$, 
$T = \bigoplus_{i,j,l \in \mathcal{I}} \mathcal{S}(l,i\otimes j)$,
$\alpha$, $\bar\alpha$ are defined via the associator of $\mathcal{S}$, and
$\psi(\id_i) = (\dim i)^{1/2} \cdot \id_i$, $\phi = \Dim(\mcS)^{-1}$
(see Section~\ref{Drinfeld-example}).
\item
\label{eg2}
Here, $\mathcal{C}$ is again an arbitrary MFC and we set $T=A$, $\alpha = \bar\alpha = \Delta \circ \mu$ and $\psi = \id_A$, $\phi = 1$
(see Section~\ref{sec:locmod}).

\item
\label{eg3}
Choose a simple object $m_g \in \mathcal{B}_g$ for each $g \in G$. Then one can take $A = \bigoplus_{g \in G} m_g^* \otimes m_g$, $T = \bigoplus_{g,h \in G} m_{gh}^* \otimes m_g \otimes m_h$, and
$\psi \big|_{m_g^* \otimes m_g} = (\dim m_g)^{-1/2} \id$, $\phi = 1/|G|$
(see \cite[Sec.\,5]{CRS3}).
\end{enumerate}

Fix an orbifold datum $\mathbb{A}$ in $\mathcal{C}$.
We now turn to the main construction in this paper, that of the category $\mathcal{C}_\mathbb{A}$. Its objects are tuples $(M,\tau_1,\tau_2,\overline{\tau_1},\overline{\tau_2})$, where $M$ is an $A$-$A$-bimodule, $\tau_i : M \otimes_{A} T \to T \otimes_{A,i} M$ for $i=1,2$ are $A$-$AA$-bimodule maps, and
$\overline{\tau_i}$ is 
-- up to normalisation -- inverse to $\tau_i$. The $\tau_i$, $\overline{\tau_i}$ satisfy conditions listed in Section \ref{subsec:CA_defined}. Morphisms in $\mathcal{C}_\mathbb{A}$ are bimodule maps compatible with the $\tau_i$ and $\overline{\tau_i}$.

We endow $\mathcal{C}_\mathbb{A}$ with the structure of a ribbon category (see Section \ref{subsec:CA_defined}). For example, the tensor product of two objects $(M,\text{$\tau$'s})$ and $(N,\text{$\tau$'s})$ has as underlying bimodule simply $M \otimes_A N$.
We call the orbifold datum $\opA$ \it{simple} if the tensor unit $\opid_{\mcC_\opA} := A$ is a simple object in $\mcC_\opA$.
We stress that $A$ can be simple in $\mcC_\opA$ even if it is not simple as a bimodule over itself.
We show (Theorem \ref{thm:CA_is_modular}):

\begin{thm}
For $\mathcal{C}$ a MFC and $\mathbb{A}$ a simple orbifold datum in $\mathcal{C}$, $\mathcal{C}_\mathbb{A}$ is also a MFC.
\end{thm}

Recall that the dimension of a MFC is defined as the sum over the squares of the quantum dimensions of
simple objects. The dimension of $\mcC_\opA$ can be directly expressed in terms of the constituents of the orbifold datum $\opA$. Namely, we show that $\tr_\mcC \psi^4$ is non-zero and that
\begin{equation}\label{intro-DimCA}
\Dim \mcC_\opA \,=\, \frac{\Dim\mcC}{\phi^4 \cdot (\tr_\mcC \psi^4)^2} \ .
\end{equation}
Let us also illustrate this in the first two examples. 
For example~\ref{eg1}, $\Dim\mcC=1$, $\phi = \Dim(\mcS)^{-1}$ and $\tr_\mcC \psi^4=\Dim(\mcS)$
and so $\Dim \mcC_\opA = (\Dim\mcS)^2$.
For example~\ref{eg2} we have $\phi=1$ and $\tr_\mcC \psi^4=\dim(A)$, so that $\Dim \mcC_\opA = \Dim\mcC / \dim(A)^2$. 
Both results are to be expected in light of the following theorem (Theorems \ref{thm:CA_equiv_CAloc} and \ref{thm:D_equiv_centre}).
\begin{thm}
For the orbifold datum in example~\ref{eg1} we have the equivalence of ribbon categories $\mathcal{C}_\mathbb{A} \cong \mathcal{Z}(\mathcal{S})$, and for that in example~\ref{eg2} we have $\mathcal{C}_\mathbb{A} \cong \mathcal{C}_A^\mathrm{loc}$.
\end{thm}

This provides a unified proof of modularity of $\mathcal{Z}(\mathcal{S})$ and of $\mathcal{C}_A^\mathrm{loc}$.
In the third example, the evident conjecture is that  $\mathcal{C}_\mathbb{A} \cong \mathcal{B}^G$, but we do not treat this here.
Indeed, in this case $\phi = 1/|G|$, $\tr_\mcC \psi^4=|G|$ and the dimension formula~\eqref{intro-DimCA} gives $\Dim \mcC_\opA = \Dim(\mathcal{B}_e) \,  |G|^2$, as expected.
Further support is given in part 2 of the next remark.

\begin{rem}~
\label{rem:Witt_equiv}
\begin{enumerate}
	\setlength{\leftskip}{-1em}
\item
Examples 2 and 3 can also be obtained by a construction using Hopf monads developed in \cite{CZW}, but example 1 is in general not covered by that construction.		
\item
In \cite{KZ} an enriched version of the Drinfeld centre was introduced.\footnote{
We are grateful to David Penneys and David Reutter for bringing the enriched centre to our attention and for explaining the relation to $\mathcal{C}_\mathbb{A}$.}
Let $\mathcal{C}$, $\mathcal{D}$ be fusion categories and let $\mathcal{C}$ be in addition braided. Let $F: \mathcal{C} \to \mathcal{Z}(\mathcal{D})$ be a braided functor. In the case that $\mathcal{C}$ is a MFC, as we consider here, the enriched centre is  $F(\mathcal{C})'$, i.e.\ the commutant of the image of $\mathcal{C}$ in $\mathcal{Z}(\mathcal{D})$. The constructions 1--3 mentioned in the beginning are all instances of enriched centres. For case 1 this is trivial, and for cases 2 and 3 this is established by the following factorisations of Drinfeld centres
\begin{equation}
\mathcal{Z}(\mathcal{C}_A) \cong \mathcal{C}^{\mathrm{rev}} \boxtimes \mathcal{C}_A^\mathrm{loc}
\qquad , \qquad 
\mathcal{Z}(\mathcal{B}) \cong (\mathcal{B}_e)^{\mathrm{rev}} \boxtimes \mathcal{B}^G \ ,
\end{equation}
where $(-)^\mathrm{rev}$ refers to the category with inverse braiding and twist,
see \cite[Cor.\ 3.30]{DMNO} and \cite[Thm.\,2]{CGPW}.
It is therefore expected that $\mathcal{C}_\mathbb{A}$ is an enriched Drinfeld centre in general.
Conjecturally, the relevant category $\mathcal{D}$ which satisfies 
$\mathcal{Z}(\mathcal{D}) \cong \mathcal{C}^{\mathrm{rev}} \boxtimes \mathcal{C}_\mathbb{A}$ is defined similarly to $\mathcal{C}_\mathbb{A}$, but objects now are triples $(M,\tau_1,\overline{\tau_1})$ which correspondingly satisfy fewer conditions, cf.\ Remark~\ref{rem:D-Interface-Wilson}.
In particular, $\mathcal{D}$ itself is no longer braided.
This will be elaborated	in a separate paper \cite{M}.

Note that proving this equivalence would give as a corollary that $\mathcal{C}$ and $\mathcal{C}_\mathbb{A}$ are Witt equivalent \cite{DMNO}.
\end{enumerate}
\end{rem}

\subsection{Motivation from three-dimensional TQFT}
\label{subsec:TQFT_motivation}

Given a MFC $\mathcal{C}$, the Reshetikhin-Turaev construction \cite{RT2, Tu} provides a 3d TQFT $Z^\mathrm{RT}_\mathcal{C}$, i.e.\ a symmetric monoidal functor
\begin{equation}
Z_{\mathcal{C}}^{\mathrm{RT}} : 
\widehat{\mathrm{Bord}}{}_3(\mathcal{C})
\rightarrow 
\mathrm{Vect}_\Bbbk \ .
\end{equation}
The source category is that of three-dimensional bordisms with embedded $\mathcal{C}$-coloured ribbon graphs, and the hat denotes a certain extension needed to absorb a glueing anomaly, see \cite{Tu} for details.

One can extend $Z_{\mathcal{C}}^{\mathrm{RT}}$ to a larger bordism category $\widehat{\mathrm{Bord}}{}_3^\mathrm{def}(\mathcal{C})$ of stratified bordisms \cite{CRS2}, where the various strata are labelled by algebraic data in $\mathcal{C}$: 3-strata are unlabelled, or, equivalently, all labelled by the MFC $\mathcal{C}$; 2-strata, aka.\ surface defects, are labelled by $\Delta$-separable symmetric Frobenius algebras in $\mathcal{C}$ \cite{KS, FSV, CRS2}; 1-strata are labelled by (bi)modules over an appropriate tensor product of these algebras and 0-strata by the corresponding intertwiners, see \cite{CRS2} for details.

Starting from such a TQFT on stratified manifolds one can introduce the {\em generalised orbifold construction} \cite{CRS1,CRS3}. The idea is to carry out a state sum construction internal to the given TQFT. 
Roughly speaking, given an orbifold datum $\mathbb{A}$ one picks a simplicial decomposition of a bordism, passes to the Poincar\'e dual cell decomposition, and decorates each 2-stratum by $A$, each $1$-stratum by $T$ and each $0$-stratum by $\alpha$ or $\bar\alpha$, depending on orientations. Finally, in each $2$-cell one inserts
$\psi$ and in each 3-cell one inserts $\phi$. The conditions on $\mathbb{A}$ ensure that the value of $Z_{\mathcal{C}}^{\mathrm{RT}}$ on such a stratified bordism is independent of the choice of simplicial decomposition. It is shown in \cite{CRS3} that one obtains a new 3d TQFT, called the {\em generalised orbifold},
\begin{equation}
Z_{\mathcal{C}}^{\mathrm{orb},\mathbb{A}} : 
\widehat{\mathrm{Bord}}{}_3
\rightarrow 
\mathrm{Vect}_\Bbbk \ ,
\end{equation}
which, at least at this point, is only defined on bordisms without embedded  ribbon graphs.

The name ``generalised orbifold'' derives from the observation that example 3 is an actual orbifold by $G$, but that the same construction also covers examples 1 and 2. This is similar to the use of the term ``generalised symmetries'' in \cite{CZW}.

\medskip

\begin{figure}
\centering
\begin{subfigure}[b]{0.2\textwidth}
	\centering
	\pic[1.5]{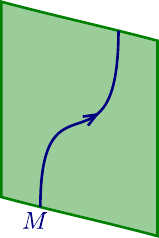}
	\caption{}
	\label{fig:obj}
\end{subfigure}
\begin{subfigure}[b]{0.5\textwidth}
	\centering
	\pic[1.5]{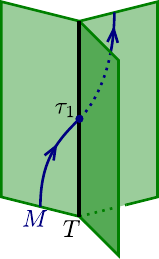} \pic[1.5]{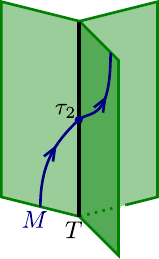}
	\caption{}
	\label{fig:Tcross}
\end{subfigure}\\
\begin{subfigure}[b]{1.0\textwidth}
	\centering
	\pic[1.5]{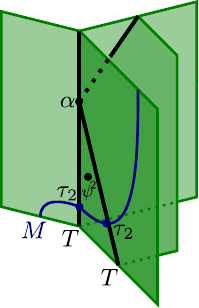} $=$ \pic[1.5]{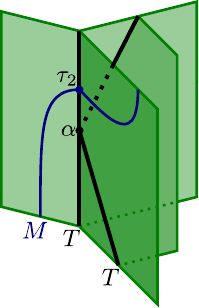}
	\caption{}
	\label{fig:axiom}
\end{subfigure}\\
\begin{subfigure}[b]{0.2\textwidth}
	\centering
	\pic[1.5]{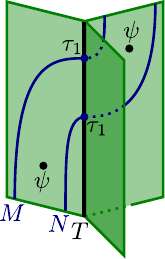}
	\caption{}
	\label{fig:prod}
\end{subfigure}
\begin{subfigure}[b]{0.4\textwidth}
	\centering
	\pic[1.5]{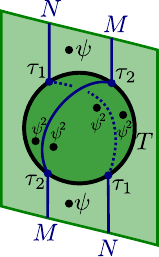} $\cdot ~\phi$
	\caption{}
	\label{fig:braiding}
\end{subfigure}
\caption{(a) Line defects in a surface defect are labelled by bimodules $M$. (b) $T$-crossings $\tau_1$, $\tau_2$ into adjacent surface defects. (c) Example of a compatibility condition between the $T$-crossings $\tau_i$ and $\alpha$. (d) Example of the $T$-crossing for the tensor product $M \otimes_A N$ of Wilson line defects. (e) Braiding of two Wilson line defects via a surface defect bubble.
\\
The various insertions of $\psi$ and $\phi$ are result from the orbifold construction and can be ignored at first.
 }
\label{fig:line_stuff}
\end{figure}

Two natural questions are now whether $Z_{\mathcal{C}}^{\mathrm{orb},\mathbb{A}}$ is equivalent to a Reshetikhin-Turaev type TQFT $Z_{\mathcal{C}'}^{\mathrm{RT}}$ for some other MFC $\mathcal{C}'$, and, if so, how to obtain $\mathcal{C}'$ from $\mathcal{C}$ and $\mathbb{A}$. The second question was the motivation to undertake the research presented here. Indeed, we conjecture that $\mathcal{C}' = \mathcal{C}_\mathbb{A}$, and will return to this question in a future publication.

The idea to extract $\mathcal{C}'$ is to investigate how to describe ribbon graphs -- or in other words Wilson lines and their junctions -- in the generalised orbifold TQFT. We will think of such ribbon graphs as line defects embedded in the 2-strata of orbifold stratification. By \cite{CRS2}, such line defects are given by $A$-$A$-bimodules (Figure~\ref{fig:line_stuff}\,a). The line defects need to be able to cross $T$-labelled 1-strata into adjacent 2-strata. Such junctions are described precisely by the data $\tau_i$, $\overline{\tau_i}$ (Figure~\ref{fig:line_stuff}\,b). To achieve independence of the initial simplicial decomposition, one has to be able to slide line defects across junction points of $T$-defects in various ways, for example as in Figure~\ref{fig:line_stuff}\,c. 
The tensor product is given by placing two line defects parallel to each other, and the $T$-crossings are described in this way, too (Figure~\ref{fig:line_stuff}\,d). Finally, the braiding is obtained by inserting a bubble on a 2-stratum and using this to make one line-defect pass over another (Figure~\ref{fig:line_stuff}\,e).

In the main text we will not use the connection between the algebraic definition of $\mathcal{C}_\mathbb{A}$ and line defects in the generalised orbifold TQFT, but the entire construction was found by exploiting this relation.

\medskip

This paper is organised as follows. In Section~\ref{sec:pre} we recall some basic definitions and facts about algebras and modules, and then give the definition of an orbifold datum $\opA$. Section~\ref{sec:Wilson} contains the
definition of the category $\mcC_\opA$ and of its ribbon structure, as well as our main theorem that $\mcC_\opA$ is a modular fusion category. In Section~\ref{sec:examples} we present two examples in detail:  local modules and the Drinfeld centre. 
Some useful identities to work with the category $\mcC_\opA$ and a variant of the monadicity theorem and  are collected in the Appendix.

\subsubsection*{Acknowledgements}

We would like to thank
Nils Carqueville,
David Penneys,
David Reutter,
Iordanis Romaidis,
Gregor Schaumann,
Daniel Scherl,
Christoph Schweigert, and
Zhenghan Wang
for helpful discussions.
We are grateful to Nils Carqueville for his careful reading of a draft of this paper.
VM is supported by the Deutsche Forschungsgemeinschaft (DFG) via the Research Training Group RTG~1670.
IR is partially supported by the DFG via the RTG~1670 and the Cluster of Excellence EXC~2121.

\section{Orbifold data}\label{sec:pre}

In this section we will recall the definition of an orbifold datum from \cite{CRS3}. To do so, we first list our conventions for spherical and modular fusion categories and summarise the definition of Frobenius algebras and their modules in a tensor category.

\subsection{Conventions}
\label{subsec:conventions}
Let $\opk$ be an algebraically closed field.
For a spherical fusion category $\mcS$ over $\opk$, an object $X\in\mcS$ and a morphism $f:X\ra X$, 
denote by $\tr_\mcS f\in\opk$ the trace of $f$ and by $|X|_\mcS = \tr_\mcS \id_X$ the categorical dimension of $X$.
$\Irr_\mcS$ will always denote a set of representatives of isomorphism classes of simple objects of $\mcS$ and we will assume that the tensor unit $\opid\in\mcS$ is in $\Irr_\mcS$.
The \it{dimension} of $\mcS$ is the sum 
$\Dim \mcS := \sum_{k\in\Irr_\mcS} (|k|_\mcS)^2$.
For $k \in \Irr_\mcS$ we have $|k|_\mcS \neq 0$, see \cite[Prop.\ 4.8.4]{EGNO}. If in addition the characteristic of $\opk$ is zero, then automatically also $\Dim \mcS \neq 0$  \cite[Thm.\ 2.3]{ENO}, but in non-zero characteristic it can happen that $\Dim \mcS = 0$ \cite[Sec.\ 9.1]{ENO}.

Throughout the entire paper we will extensively use string diagram notation.
Diagrams are to be read from bottom to top, and a downwards directed strand represents the dual of an object.
The evaluation/coevaluation morphisms of an object $X\in\mcS$ will be denoted by:
\begin{align}
&[\ev_X: X^* \otimes X \ra \opid] = \pic[1.5]{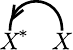},
&&[\coev_X: \opid \ra X \otimes X^*] = \pic[1.5]{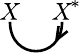}, \nonumber\\
&[\evt_X: X \otimes X^* \ra \opid] = \pic[1.5]{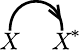},
&&[\coevt_X: \opid \ra X^* \otimes X] = \pic[1.5]{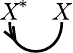}.
\end{align}
Labels for objects and morphisms will be omitted whenever they are clear from the context.

Let $\mcC$ be a braided fusion category.
For all $X,Y\in\mcC$, the braiding morphisms $c_{X,Y}: X \otimes Y \ra Y \otimes X$, and their inverses will be depicted by:
\begin{equation}
c_{X,Y}       = \pic[1.5]{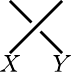}~, \qquad
c^{-1}_{X,Y} = \pic[1.5]{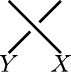}~.
\end{equation}

A braided spherical fusion category $\mcC$ is automatically ribbon (see e.g.\ \cite[Lem.4.5]{TuVi}), with the twist morphism of an object $X\in\mcC$ built out of braiding and duality morphisms as
\begin{equation}
\label{eq:twist_def}
\theta_X := \pic[1.5]{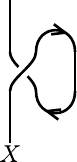} = \pic[1.5]{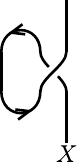}, \quad
\theta^{-1}_X := \pic[1.5]{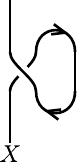} = \pic[1.5]{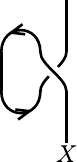}.
\end{equation}

A ribbon fusion category is also called \it{premodular}.
A premodular category $\mcC$ is called \it{modular} if the matrix
\begin{equation}
s_{i,j} := 
\tr_\mcS( c_{j,i} \circ c_{i,j} ) =
\pic[1.5]{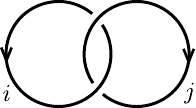} ~, \qquad i,j\in\Irr_\mcC
\end{equation}
is invertible.

We fix a modular fusion category $\mcC$ over $\opk$.
For notational simplicity, $\mcC$ will be assumed to have strict monoidal and pivotal structures (without loss of generality \cite{ML,NS}), 
and the symbol $\otimes$ for the monoidal product will sometimes be omitted.
\subsection{Frobenius algebras and their modules}
\label{subsec:FAs_and_mods}

In this section we briefly recall the notion of Frobenius algebras and their modules in $\mcC$, more details can be found e.g.\ in \cite{FRS1}.

A \it{Frobenius algebra} in $\mcC$ is a tuple
\begin{equation}
A\in\mcC,
\quad \mu:A\otimes A \ra A, \quad \eta:\opid \ra A,
\quad\D: A \ra A \otimes A, \quad \vareps: A \ra \opid,
\end{equation}
where $(A,\mu,\eta)$ is an associative unital algebra and $(A,\D,\vareps)$ is a coassociative counital coalgebra, such that
\footnote{In the online version, $A$-coloured strands are drawn in green.}
\begin{equation}
\pic[1.5]{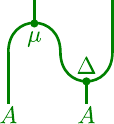} = \pic[1.5]{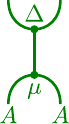} = \pic[1.5]{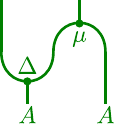}.
\end{equation}
The unit (resp.\ counit) will be denoted by \pic[1.5]{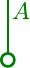} (resp.\ \pic[1.5]{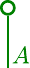}).
A Frobenius algebra is called
\begin{equation}
\label{eq:D_sep_and_symm}
\text{\it{symmetric} if }
\pic[1.5]{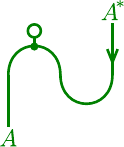} = \pic[1.5]{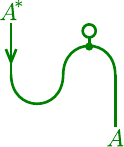}
\text{ and $\D$-\it{separable} if }
\pic[1.5]{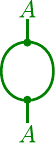} = \pic[1.5]{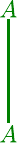}.
\end{equation}

A (\it{left-}) \it{module} of a Frobenius algebra $A$ is a module $(M\in\mcC, ~\l: A \otimes M \ra M)$ of the underlying algebra.
It is simultaneously a comodule with the coaction given by $[M \xra{(\D\circ\eta) \otimes \id_M} A \otimes A \otimes M \xra{\id_A\otimes\l} A \otimes M]$, or in graphical notation
\begin{equation}
\pic[1.5]{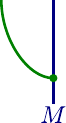} := \pic[1.5]{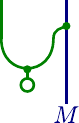}.
\end{equation}
One easily generalises this to right modules and bimodules of a Frobenius algebra.

Let $\mcACA$ be the category of $A$-$A$-bimodules.
As usual, it is a tensor category with the monoidal product $M \otimes_A N$ of $M,N \in \mcACA$ given by
\begin{equation}
M \otimes_A N := \coker[M\otimes A \otimes N \xra{\rho_M \otimes \id_N - \id_M \otimes \l_N} M \otimes N],
\end{equation}
where $\rho_M$ and $\l_N$ are the corresponding right and left actions.
By \cite[Prop.\,5.24, Rem.\,5.25]{FS}\footnote{
	\cite{FS} uses a \it{special} Frobenius algebra $A$, which is the same as $\D$-separable, but with the extra assumption $\dim_\mcC A \neq 0$; for this particular result, this assumption is not necessary.}
we have:
\begin{prp}
\label{prp:ACA_is_ssi}
For $A\in\mcC$ a $\D$-separable Frobenius algebra, $\mcACA$ is a finitely semisimple
\footnote{
By ``finitely semisimple'' we mean that there are finitely many isomorphism classes of simple objects, and that each object is isomorphic to a finite direct sum of simple objects.\label{fn:finitely-ssi}
}	 
monoidal category.
\end{prp}

If $A$ is a $\D$-separable Frobenius algebra, $M\otimes_A N$ is isomorphic to the image of the idempotent
\begin{equation}
\label{eq:MxAN_idemp}
\pic[1.5]{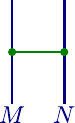} := \pic[1.5]{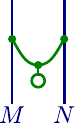}.
\end{equation}
The projection $\pi:M\otimes N \ra M\otimes_A N$ has a section $\imath$ in $\mcACA$, and $\pi$ and $\imath$ then split the idempotent in \eqref{eq:MxAN_idemp}. The graphical notation we will use is
\begin{equation}
\pi = \pic[1.5]{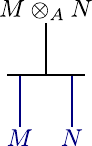}
\quad\text{and}\quad
\imath = \pic[1.5]{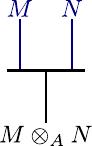}.
\end{equation}
For any object $X\in\mcC$, a morphism $f:X \ra M \otimes_A N$ or $g: M \otimes_A N \ra X$ can be uniquely given by morphisms $\widehat{f}:X \ra M \otimes N$, $\widehat{g}: M \otimes N \ra X$, such that
\begin{equation}
\label{eq:proj_swallow}
\pic[1.5]{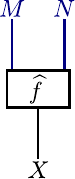} = \pic[1.5]{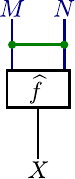}, \qquad
\pic[1.5]{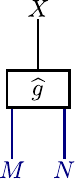} = \pic[1.5]{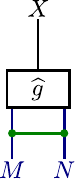},
\end{equation}
namely $\widehat{f} := \imath \circ f$ and $\widehat{g} := g \circ \pi$.
By abuse of notation, the overhats like in \eqref{eq:proj_swallow} will be omitted in the following.

If $A,B\in\mcC$ are $\D$-separable Frobenius algebras, so is $A\otimes B$, where we equip it with
\begin{equation}
\text{product:}    \pic[1.5]{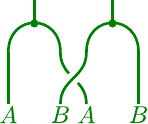},
\text{ coproduct:} \pic[1.5]{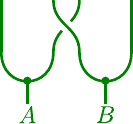},
\text{ unit:}	   \pic[1.5]{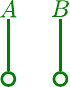},
\text{ counit:}	   \pic[1.5]{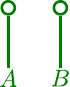}.
\end{equation}
An $A\otimes B$ module $M$ is the same as a simultaneous $A$ and $B$ module such that
\begin{equation}
\label{eq:AxB_mod_comm_acts}
\pic[1.5]{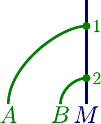} = \pic[1.5]{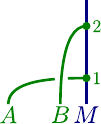}
\ .
\end{equation}
Here and for the rest of the paper we indicate the action of the first and second
tensor factor by indices $1,2$.

Let $M$ be a right $AB$-module and $K$, $L$ be left $A$- and $B$-modules respectively.
We define the following partial tensor products
\begin{equation}
M \otimes_1 K := \im \pic[1.5]{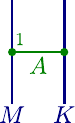}, \quad
M \otimes_2 L := \im \pic[1.5]{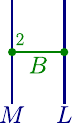},
\end{equation}
where the horizontal lines denote the idempotents as in \eqref{eq:MxAN_idemp}.
Note that $M\otimes_1 K$ (resp.\ $M\otimes_2 L$) is a right $B$-module (resp.\ right $A$-module).

We will often encounter $A$-$AA$-bimodules and their partial tensor products.
In this case, the left action will be indicated by $0$ whenever it is necessary to avoid ambiguity.
The two right actions will be distinguished by indices $1$, $2$, like in \eqref{eq:AxB_mod_comm_acts}.
For an $A$-$AA$-bimodule $M$, the dual $M^*$ is an $AA$-$A$-bimodule, $M\otimes_1 M$, $M \otimes_2 M$ are $A$-$AAA$ bimodules, $M^* \otimes_0 M$ is an $AA$-$AA$-bimodule, etc.
\subsection{Orbifold data for a MFC}
\label{subsec:orb_data}
\begin{figure}
	\captionsetup[subfigure]{labelformat=empty}
	\centering
	\begin{subfigure}[b]{0.5\textwidth}
		\centering
		\pic[1.5]{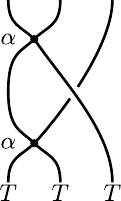}$=$\pic[1.5]{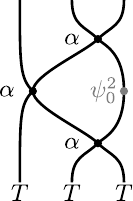}
		\caption{}
		\label{eq:O1}
	\end{subfigure}\hspace{-2em}\raisebox{5.5em}{(O1)}\\
	\begin{subfigure}[b]{0.4\textwidth}
		\centering
		\pic[1.5]{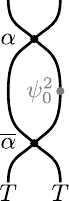}$=$\pic[1.5]{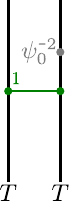}
		\caption{}
		\label{eq:O2}
	\end{subfigure}\hspace{-2em}\raisebox{5.5em}{(O2)}
	\begin{subfigure}[b]{0.4\textwidth}
		\centering
		\pic[1.5]{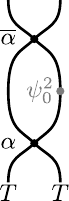}$=$\pic[1.5]{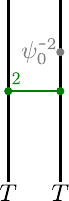}
		\caption{}
		\label{eq:O3}
	\end{subfigure}\hspace{-2em}\raisebox{5.5em}{(O3)}\\
	\begin{subfigure}[b]{0.4\textwidth}
		\centering
		\pic[1.5]{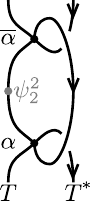}$=$\pic[1.5]{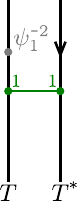}
		\caption{}
		\label{eq:O4}
	\end{subfigure}\hspace{-2em}\raisebox{5.5em}{(O4)}
	\begin{subfigure}[b]{0.4\textwidth}
		\centering
		\pic[1.5]{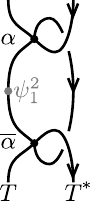}$=$\pic[1.5]{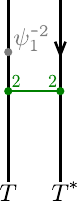}
		\caption{}
		\label{eq:O5}
	\end{subfigure}\hspace{-2em}\raisebox{5.5em}{(O5)}\\
	\begin{subfigure}[b]{0.4\textwidth}
		\centering
		\pic[1.5]{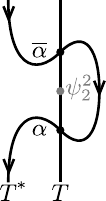}$=$\pic[1.5]{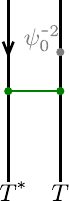}
		\caption{}
		\label{eq:O6}
	\end{subfigure}\hspace{-2em}\raisebox{5.5em}{(O6)}
	\begin{subfigure}[b]{0.4\textwidth}
		\centering
		\pic[1.5]{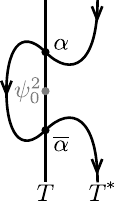}$=$\pic[1.5]{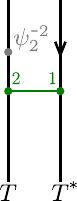}
		\caption{}
		\label{eq:O7}
	\end{subfigure}\hspace{-2em}\raisebox{5.5em}{(O7)}\\
	\begin{subfigure}[b]{0.8\textwidth}
		\centering
		\pic[1.5]{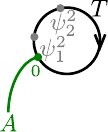}$=$\pic[1.5]{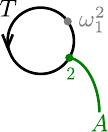}$=$\pic[1.5]{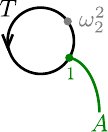}$=$\pic[1.5]{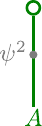}$\cdot \phi^{-1}$
		\caption{}
		\label{eq:O8}
	\end{subfigure}\hspace{-2em}\raisebox{4.25em}{(O8)}

\vspace*{-1em}

	\caption{Identities an orbifold datum has to satisfy. 
		All these string diagrams are drawn in $\mcC$, but by the comment below \eqref{eq:proj_swallow} they induce identities also between the appropriate tensor products over $A$ in $\mcACA$.
		\\
		To manipulate expressions involving orbifold data it is best to first ignore all the actions of $\psi$, which is why we draw them in grey.}
	\label{fig:orb_ids}
\end{figure}
Here we recall the definition of an orbifold datum as given in 
\cite[Def.\,3.4]{CRS3} 
(where it is called a ``special orbifold datum'').
\begin{defn}
\label{def:orb_data}
An \it{orbifold datum} in $\mcC$ is a tuple $\opA = (A, T, \a, \abar, \psi, \phi)$, where
\begin{itemize}
\item $A$ is a symmetric $\D$-separable Frobenius algebra in $\mcC$;
\item $T$ is an $A$-$AA$ bimodule in $\mcC$;
\item $\a:T \otimes_2 T \ra T \otimes_1 T$, $\abar:T \otimes_1 T \ra T \otimes_2 T$ are $A$-$AAA$ bimodule morphisms;
\item $\psi:A \ra A$ is an invertible $A$-$A$-bimodule morphism;
\item $\phi \in \opk^\times$;
\end{itemize}
These are subject to conditions \eqrefO{1}--\eqrefO{8} in Figure \ref{fig:orb_ids}, where the following notation is used:
\begin{equation}
\pic[1.5]{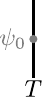}   := \pic[1.5]{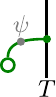}, 	\quad
\pic[1.5]{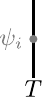}   := \pic[1.5]{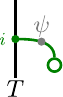}, 	\quad
\pic[1.5]{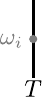} := \pic[1.5]{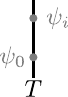},   \quad
i = 1,2.
\end{equation}
\end{defn}

As mentioned in the introduction, \cite{CRS3} provides three examples of orbifold data.
We will look at two of them in detail in Section \ref{sec:examples}.
\section{The category of Wilson lines}\label{sec:Wilson}
In this section, we give the main definition of this paper and prove our main result, namely we define the category $\mcC_\opA$ of Wilson lines and show that it is a modular fusion category.

\medskip

For the remainder of this section, 
we fix a MFC $\mcC$ and an orbifold datum $\opA = (A, T, \a, \abar, \psi, \phi)$  in $\mcC$.
\subsection{Definition of $\mcC_\opA$ as linear category}
\label{subsec:CA_defined}
\begin{figure}
	\captionsetup[subfigure]{labelformat=empty}
	\centering
	\begin{subfigure}[b]{0.45\textwidth}
		\centering
		\pic[1.5]{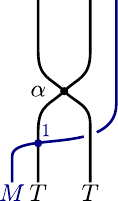}$=$\pic[1.5]{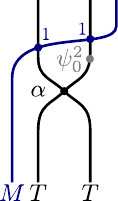}
		\caption{}
		\label{eq:T1}
	\end{subfigure}\hspace{-2em}\raisebox{5.5em}{(T1)}
	\begin{subfigure}[b]{0.45\textwidth}
		\centering
		\pic[1.5]{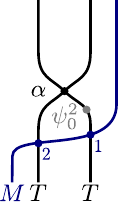}$=$\pic[1.5]{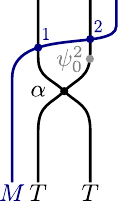}
		\caption{}
		\label{eq:T2}
	\end{subfigure}\hspace{-2em}\raisebox{5.5em}{(T2)}\\
	\begin{subfigure}[b]{0.45\textwidth}
		\centering
		\pic[1.5]{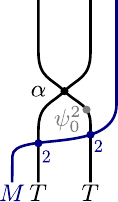}$=$\pic[1.5]{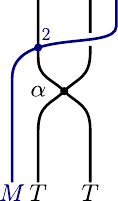}
		\caption{}
		\label{eq:T3}
	\end{subfigure}\hspace{-2em}\raisebox{5.5em}{(T3)}\\
	\begin{subfigure}[b]{0.35\textwidth}
		\centering
		\pic[1.5]{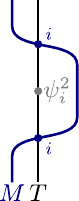}$=$\pic[1.5]{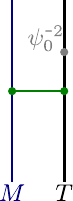}
		\caption{}
		\label{eq:T4}
	\end{subfigure}\hspace{-1em}\raisebox{5.5em}{(T4)}
	\begin{subfigure}[b]{0.35\textwidth}
		\centering
		\pic[1.5]{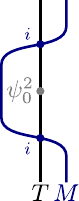}$=$\pic[1.5]{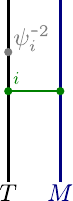}
		\caption{}
		\label{eq:T5}
	\end{subfigure}\hspace{-1em}\raisebox{5.5em}{(T5)}\hspace{4em}\raisebox{5.5em}{$i=1,2$}\\
	\begin{subfigure}[b]{0.35\textwidth}
		\centering
		\pic[1.5]{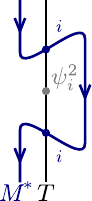}$=$\pic[1.5]{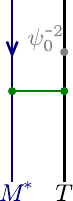}
		\caption{}
		\label{eq:T6}
	\end{subfigure}\hspace{-1em}\raisebox{5.5em}{(T6)}
	\begin{subfigure}[b]{0.35\textwidth}
		\centering
		\pic[1.5]{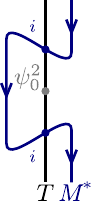}$=$\pic[1.5]{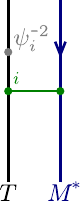}
		\caption{}
		\label{eq:T7}
	\end{subfigure}\hspace{-1em}\raisebox{5.5em}{(T7)}\hspace{4em}\raisebox{5.5em}{$i=1,2$}
	\caption{Identities for an object of $\mcC_\opA$.}
	\label{fig:obj_ids}
\end{figure}
\begin{defn}
\label{def:CA}
Define the category $\mcC_\opA$ to have:
\begin{itemize}
\item
\it{Objects}: tuples $(M, \tau_1, \tau_2, \overline{\tau_1}, \overline{\tau_2})$, where
\begin{itemize}
\item $M$ is an $A$-$A$ bimodule;
\item
$\tau_1: M \otimes_0 T \ra T \otimes_1 M$,
$\tau_2: M \otimes_0 T \ra T \otimes_2 M)$,
$\overline{\tau_1}: T \otimes_1 M \ra M \otimes_0 T$,
$\overline{\tau_2}: T \otimes_2 M \ra M \otimes_0 T$
are $A$-$AAA$-bimodule morphisms, denoted by
\begin{equation}
\tau_i = \pic[1.5]{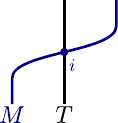}, \qquad
\overline{\tau_i} := \pic[1.5]{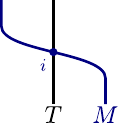}, \qquad i = 1,2 ~,
\end{equation}
such that the identities in Figure \ref{fig:obj_ids} are satisfied.
(Recall from the end of Section~\ref{subsec:FAs_and_mods} that the notation $\otimes_0$ refers to the left-$A$-action on $T$.)
\end{itemize}
\item \it{Morphisms}: 
A morphism
$f: (M, \tau_1^M, \tau_2^M, \overline{\tau_1^M}, \overline{\tau_2^M}) \ra (N, \tau_1^N, \tau_2^N, \overline{\tau_1^N}, \overline{\tau_2^N})$ is an $A$-$A$-bimodule morphism $f: M \ra N$, such that 
$\tau_i^N\circ (f \otimes_0 \id_T) = (\id_T \otimes_i f)\circ \tau_i^M $, 
$i = 1,2$, or graphically
\begin{equation}
\label{eq:M} \tag{M}
\pic[1.5]{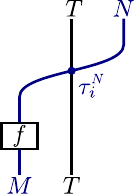} = \pic[1.5]{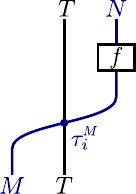}, \quad i=1,2 ~.
\end{equation}
\end{itemize}
Note that \eqrefT{4} and \eqrefT{5} imply that $\overline{\tau_i}$ is uniquely determined by $\tau_i$.
We will refer to the morphisms $\tau_1$, $\tau_2$ as \it{$T$-crossings} and to $\overline{\tau_1}$, $\overline{\tau_2}$ as their \it{pseudo-inverses}.
\end{defn}
\begin{example}
\label{ex:obj_A}
For each $\l\in\opZ$, $(A, \tau_1^\l, \tau_2^\l, \overline{\tau_1^\l}, \overline{\tau_2^\l})$ is an object of $\mcC_\opA$, where the $T$-crossings are
\begin{equation}
\tau_i^\l := \pic[1.5]{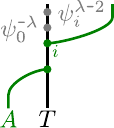}, \qquad
\overline{\tau_i^\l} := \pic[1.5]{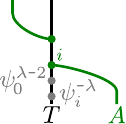}, \qquad
i = 1,2.
\end{equation}
Moreover, all these objects are isomorphic in $\mcC_\opA$.
Indeed, define a morphism
$f: (A, \tau_1^\l,  \tau_2^\l,  \overline{\tau_1^\l},  \overline{\tau_2^\l})
\ra (A, \tau_1^\mu, \tau_2^\mu, \overline{\tau_1^\mu}, \overline{\tau_2^\mu})$
to be the bimodule map $f = \psi^{\mu - \l}$.
For $i=1,2$ one has:
\begin{equation*}
\pic[1.5]{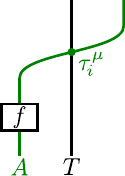}
= \pic[1.5]{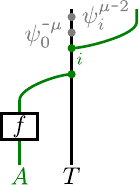}
= \pic[1.5]{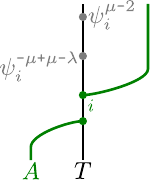}
= \pic[1.5]{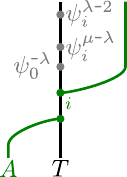}
= \pic[1.5]{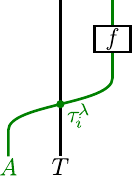},
\end{equation*}
i.e.\ $f$ is an invertible morphism in $\mcC_\opA$.
\end{example}

For $M\in\mcACA$, the following notation will be used to handle $\psi$-insertions:
\begin{equation}
\pic[1.5]{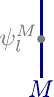} := \pic[1.5]{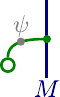}, \quad
\pic[1.5]{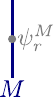} := \pic[1.5]{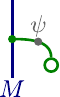}, \quad
\pic[1.5]{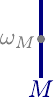} := \pic[1.5]{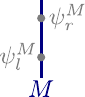}.
\end{equation}
By abuse of notation, the label $M$ in $\psi_l^M$, $\psi_r^M$, $\omega_M$ will be omitted whenever it is clear from the context.

Definitions \ref{def:orb_data} and \ref{def:CA} imply a large number of additional algebraic identities, which will be used in computations below.
We list several of them in Appendix \ref{app:identities}.
\subsection{$\mcC_\opA$ as ribbon category}
We equip $\mcC_\opA$ with the following monoidal structure:
\begin{itemize}
\item
\it{product:}
\begin{equation*}
(M, \tau_1^M, \tau_2^M, \overline{\tau_1^M}, \overline{\tau_2^M}) \otimes
(N, \tau_1^N, \tau_2^N, \overline{\tau_1^N}, \overline{\tau_2^N}) :=
(M \otimes_A N, \tau^{M,N}_1, \tau^{M,N}_2, \overline{\tau^{M,N}_1}, \overline{\tau^{M,N}_2}),
\end{equation*}
where the $T$-crossings are:
\begin{equation}\label{eq:tau^MN_i}
\tau^{M,N}_i            := \pic[1.5]{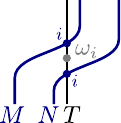}, \quad
\overline{\tau^{M,N}_i} := \pic[1.5]{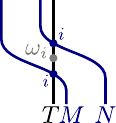}, \qquad i=1,2 ~;
\end{equation}
\item
\it{unit:} $\opid_{\mcC_\opA} = (A, \tau_1^1, \tau_2^1, \overline{\tau_1^1}, \overline{\tau_2^1})$ from Example \ref{ex:obj_A} (with the choice $\l=1$), i.e.\ with the $T$-crossings
\begin{equation}
\label{eq:unit_T-cross}
\pic[1.5]{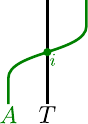} := \pic[1.5]{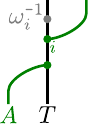}, \quad
\pic[1.5]{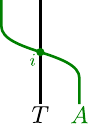} := \pic[1.5]{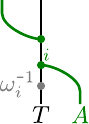}, \qquad i=1,2 ~;
\end{equation}
\item
\it{associators and unitors:} as in $\mcACA$.
\end{itemize}
One checks that $\tau^{M,N}$ satisfies the conditions \eqrefT{1}--\eqrefT{7} of a $T$-crossing.
For example, to check \eqrefT{1} for $\tau^{M,N}_1$, one first applies \eqrefT{1} for $\tau^{M}_1$ and $\tau^{N}_1$ separately. This generates two insertions of $\psi_0^2$, one of which combines with $\omega_1$ in \eqref{eq:tau^MN_i} into the two insertions of $\omega_1$ required for the two copies of $\tau^{M,N}_1$. Note that \eqrefT{1} would fail without the $\omega_i$ in \eqref{eq:tau^MN_i}.

The remaining conditions for $\mcC_\opA$ to be a monoidal category follow from those in $\mcACA$.
\begin{defn}
\label{def:simple_orb}
We call an orbifold datum $\opA$ in $\mcC$ \it{simple} if $\dim\End_{\mcC_\opA} \left(\opid_{\mcC_\opA} \right) = 1$.
\end{defn}

From now on, we will omit the $T$-crossings when referring to an object of $\mcC_\opA$.
If $M$ is an object of $\mcC_\opA$, so is the dual bimodule $M^*$, where the $T$-crossings are
\begin{equation}
\pic[1.5]{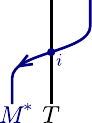}
:= \pic[1.5]{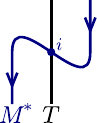}, \quad
\pic[1.5]{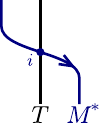}
:= \pic[1.5]{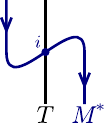} \qquad i=1,2
\end{equation}
(note that $\psi_l^{M^*} = (\psi_r^M)^*$ and $\psi_r^{M^*} = (\psi_l^M)^*$).
$M^*$ is a left and right dual of $M$ simultaneously, with evaluation/coevaluation morphisms given by:
\begin{align}
&[\ev_M: M^* \otimes_A M \ra A] := \pic[1.5]{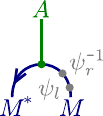},
&&[\coev_M: A \ra M \otimes_A M^*] := \pic[1.5]{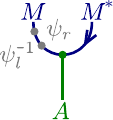}, \nonumber \\
\label{eq:evs_evts_in_D} 
&[\evt_M: M \otimes_A M^* \ra A] := \pic[1.5]{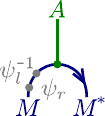},
&&[\coevt_M: A \ra M^* \otimes_A M] := \pic[1.5]{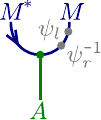}.
\end{align}
The various insertions of $\psi_l$ and $\psi_r$ are necessary to make the dualities into morphisms in $\mcC_\opA$.
The four zig-zag identities follow from those in $\mcACA$, the $\psi$-insertions cancel each other in each case.
The above duality morphisms equip $\mcC_\opA$ with pivotal structure (for that it is enough to check that the identities in \cite[Lem.\ 2.12]{CR} hold).

For a pair of objects $M,N \in \mcC_\opA$, define the morphisms
$c_{M,N}:M \otimes_A N \ra N \otimes_A M$ and $c^{-1}_{M,N}:N \otimes_A M \ra M \otimes_A N$
in $\mcACA$ as follows:
\begin{equation}
\label{eq:D_braiding}
c_{M,N} := \pic[1.5]{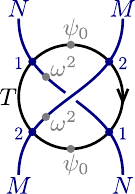} \cdot \phi
\ , \qquad
c^{-1}_{M,N} := \pic[1.5]{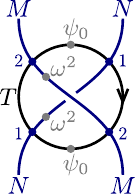} \cdot \phi
\ .
\end{equation}
That the notation $c^{-1}_{M,N}$ is indeed justified is part of the claim in Proposition~\ref{prop:c_MN-is-a-braiding} below.

\begin{lem}
\label{lem:br_ids}
For all $M,N\in\mcC_\opA$, the following identities hold:
\begin{equation}
\label{eq:br_ids}
\pic[1.5]{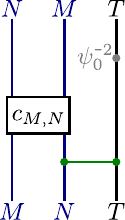} = \pic[1.5]{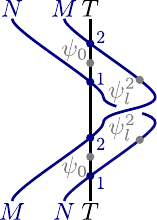}, \qquad
\pic[1.5]{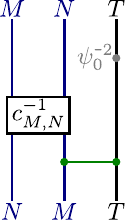} = \pic[1.5]{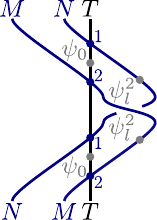}.
\end{equation}
\end{lem}
\begin{proof}
Using also the identities in Appendix \ref{app:identities}
(and which are denoted by a prime, e.g.\ \eqref{eq:T3bp})
for the first equality one has\footnote{
As already mentioned in Figure~\ref{fig:orb_ids}, in this and the following computations it is helpful to ignore the $\phi$- and $\psi$-insertions at first and only verify that these also work out as a second step.
To make this easier, all $\psi$'s are shown in grey in string diagrams.
}
\begin{align*}
&                             \pic[1.5]{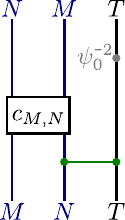} 
\stackrel{\eqrefO{6}}{=}      \pic[1.5]{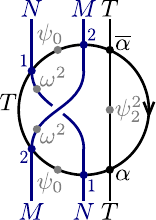} \cdot \phi
\underset{\eqrefT{1}}{\stackrel{\eqrefT{4}}{=}}
 							  \pic[1.5]{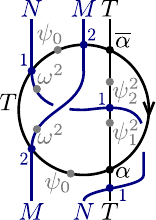} \cdot \phi\\
&\stackrel{\eqrefT{3}}{=}     \pic[1.5]{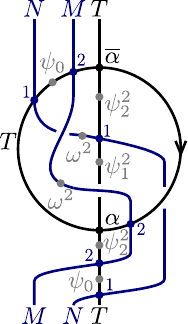} \cdot \phi
\stackrel{\eqref{eq:T3bp}}{=} \pic[1.5]{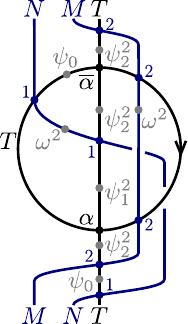} \cdot \phi
\stackrel{\eqref{eq:T1bp}}{=} \pic[1.5]{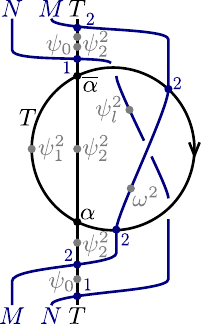} \cdot \phi\\
&\stackrel{\eqrefT{6}}{=}     \pic[1.5]{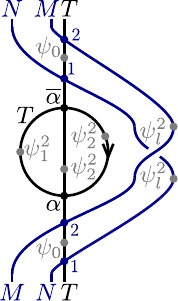} \cdot \phi
\stackrel{\eqrefO{6}}{=}      \pic[1.5]{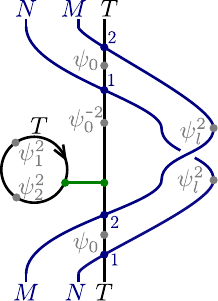} \cdot \phi
\stackrel{\eqrefO{8}}{=}      \pic[1.5]{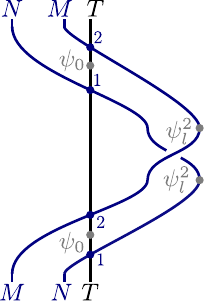}.
\end{align*}
Similarly one can show the second identity.
\end{proof}
\begin{prp}\label{prop:c_MN-is-a-braiding}
$\{c_{M,N}\}_{M,N\in\mcC_\opA}$ defines a braiding on $\mcC_\opA$.
\end{prp}
\begin{proof}
One must check that for $M,N\in\mcC_\opA$, $c_{M,N}$ and $c^{-1}_{M,N}$ are natural in $M$, $N$, satisfy the hexagon identities, are inverses of each other and that the identity \eqref{eq:M} holds.
All of this can be done by repeatedly applying \eqrefO{1}-\eqrefO{8}, \eqrefT{1}-\eqrefT{7} and \eqref{eq:M}; we only show one of the hexagon identities 
for $c_{M,N}$.
Using Lemma \ref{lem:br_ids} one gets:
\begin{align*}
&  								 \pic[1.5]{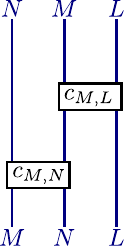}
 = 								 \pic[1.5]{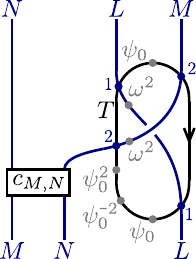} \cdot \phi
 \stackrel{\eqref{eq:br_ids}}{=} \pic[1.5]{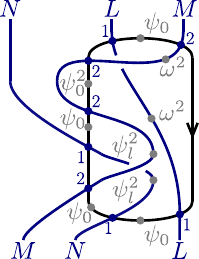} \cdot \phi
\\
&\stackrel{\eqrefT{5}}{=} 		 \pic[1.5]{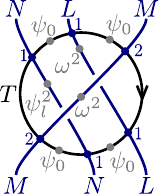} \cdot \phi
 = 								 \pic[1.5]{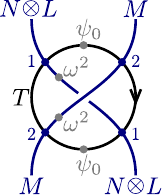} \cdot \phi
 = 								 \pic[1.5]{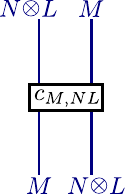}.
\end{align*}
\end{proof}
\begin{prp}
\label{prp:D_is_spherical}
$\mcC_\opA$ is spherical.
\end{prp}
\begin{proof}
One needs to check that for all $M\in\mcC_\opA$ and $f\in\End_{\mcC_\opA}$, the left and right traces of $f$ are equal, i.e.
$\evt_M\circ (f\otimes\id_{M^*}) \circ\coev_M = \ev_M\circ (\id_{M^*}\otimes f) \circ\coevt_M$.
We have:
\begin{align*}
&							           \pic[1.5]{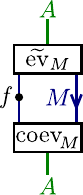}
\stackrel{\eqref{eq:evs_evts_in_D}}{=}
 							           \pic[1.5]{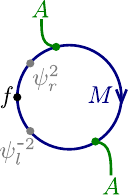}
 \stackrel{\eqrefO{8}}{=}              \pic[1.5]{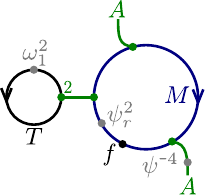} \cdot \phi
 \underset{\eqref{eq:M}}{\stackrel{\eqrefT{5}}{=}} 
\pic[1.5]{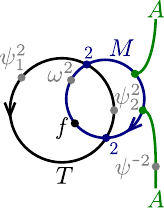} \cdot \phi\\
&\stackrel{\eqrefT{6}}{=}              \pic[1.5]{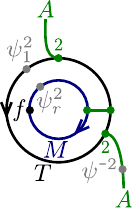} \cdot \phi
= 							           \pic[1.5]{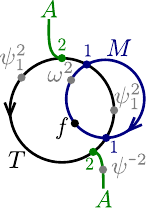} \cdot \phi
= 							           \pic[1.5]{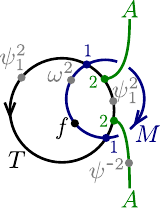} \cdot \phi\\
& \stackrel{(*)}{=} 		           \pic[1.5]{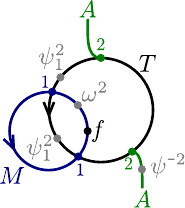} \cdot \phi
= 							           \pic[1.5]{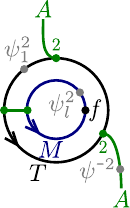} \cdot \phi
\stackrel{(**)}{=}
 							           \pic[1.5]{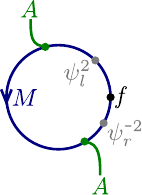}
= 							           \pic[1.5]{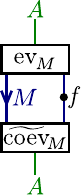},
\end{align*}
where in step $(*)$ one uses that $\mcC$ is ribbon to flip the rightmost part of the $M$-ribbon around the back of the circular $T$-ribbon (recall that all string diagrams in this section are drawn in $\mcC$).
For step $(**)$ one basically runs the first four steps in reverse order. 
\end{proof}
\begin{rem}
\begin{enumerate}
\item
For $M,N\in\mcC_\opA$ let $\Omega':M\otimes_A N \otimes_0 T \ra N \otimes_A M \otimes_0 T$ be the morphism on the left hand side of the first identity in Lemma \ref{lem:br_ids} and denote
\begin{equation}
\label{eq:almost_braiding}
\Omega := (\id_N \otimes_A \id_M \otimes_0 (\psi_2^2 \circ \psi_1^2) )) \circ \Omega' \cdot \phi
\end{equation}
Then one has $c_{M,N} = \tr_{\mcC,T} \Omega$, where $\tr_{\mcC,T}$ denotes the partial trace with respect to $T$, taken in $\mcC$.
\item
In what follows we show that $\mcC_\opA$ is
multifusion (see~\cite{EGNO} for the definition; a multifusion category with simple tensor unit is fusion).
Proposition \ref{prp:D_is_spherical} then implies that $\mcC_\opA$ is ribbon (see Section \ref{subsec:conventions}).
The twist of an object $M\in\mcC_\opA$ is obtained as in \eqref{eq:twist_def}, or explicitly:
\begin{equation}
\label{eq:twist}
\theta_M 
\underset{\eqref{eq:evs_evts_in_D}}{\stackrel{\eqref{eq:twist_def}}{:=}}
\pic[1.5]{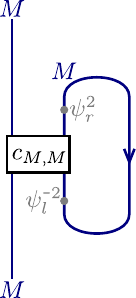} 
\stackrel{\eqref{eq:D_braiding}}{=}
\pic[1.5]{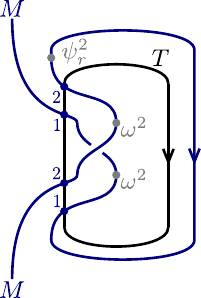} \cdot \phi.
\end{equation}
\end{enumerate}
\label{rem:mutifusion-to-ribbon}
\end{rem}
\subsection{Semisimplicity}
We will show the semisimplicity of $\mcC_\opA$ by exploiting a sequence of adjunctions.
These will involve some auxiliary categories $\mcD_1$, $\mcD_2$, which we now define.
\begin{defn}
Define the categories $\mcD_i$, $i=1,2$ as follows:
\begin{itemize}
\item
\it{objects} of $\mcD_i$ are triples $(M,\tau_i,\overline{\tau_i})$, where $M\in\mcACA$ and
$\tau_i:M \otimes_0 T \ra T \otimes_i M$ is a $T$-crossing with pseudo-inverse $\overline{\tau_i}$, i.e.\ it satisfies the identities \eqrefT{1} and \eqrefT{4}-\eqrefT{7} (for $i=1$)  and \eqrefT{3}-\eqrefT{7} (for $i=2$);
\item
\it{morphisms} of $\mcD_i$ are bimodule morphisms, satisfying the identity \eqref{eq:M} (for the given value of $i$ only).
\end{itemize}
\end{defn}
\begin{rem}\label{rem:D-Interface-Wilson}
The categories $\mcD_1$, $\mcD_2$ can be interpreted in the orbifold TQFT setting briefly outlined in Section \ref{subsec:TQFT_motivation}.
Namely, they describe Wilson lines that live on an interface between the TQFT described by $\mcC$ and the TQFT given by the orbifold datum $\opA$.
$\mcD_1$, $\mcD_2$ are also candidates for the category $\mcD$ alluded to in Remark \ref{rem:Witt_equiv}.
This will be further elaborated in \cite{M}.
\end{rem}
We will now define four functors:
\begin{equation}
	\label{eq:H_functs}
	\begin{tikzcd}
		\mcACA\arrow[r, "H_1"]\arrow[d, "H_2"]
		& \mcD_1\arrow[d, "H_{21}"] \\
		\mcD_2\arrow[r, "H_{12}"]
		& \mcC_{\opA}
	\end{tikzcd}.
\end{equation}
Namely, for any bimodule $M\in\mcACA$, define two bimodules $H_1(M)$ and $H_2(M)$ together with $T$-crossings $\tau_1^{H_1(M)}$, $\tau_2^{H_2(M)}$ and their pseudo-inverses by
\begin{align}
&H_1(M)  :=                 \im \pic[1.5]{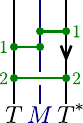},
&&H_2(M) :=                 \im \pic[1.5]{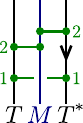},\\		\label{eq:HM_Tcross}
&\tau_1^{H_1(M)} :=             \pic[1.5]{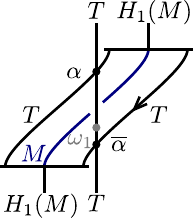},
&&\tau_2^{H_2(M)} :=            \pic[1.5]{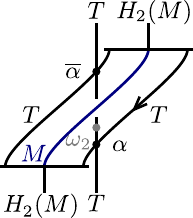},\\   \label{eq:HM_Tcross_inv}
&\overline{\tau_1^{H_1(M)}} :=  \pic[1.5]{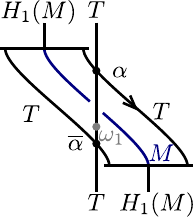},
&&\overline{\tau_2^{H_2(M)}} := \pic[1.5]{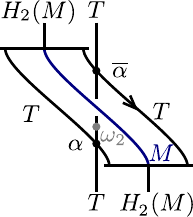}.
\end{align}
It is easy to check that $H_1(M)$, $H_2(M)$ are objects of $\mcD_1$ and $\mcD_2$, respectively.

Next, for $K\in\mcD_2$ define $H_{12}(K)$ to have the same underlying bimodule as $H_1(K)$, and set
\begin{equation}
\tau_1^{H_{12}(K)} = \tau_1^{H_{1}(K)} \quad , \quad  
\overline{\tau_1^{H_{12}(K)}} = \overline{\tau_1^{H_{1}(K)}} \ .   
\end{equation}
The $T$-crossings $\tau_2$ are defined as follows:
\begin{equation}
\tau_2^{H_{12}(K)} := \pic[1.5]{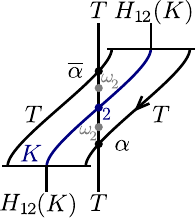} 
\quad ,\qquad
\overline{\tau_2^{H_{12}(K)}} := \pic[1.5]{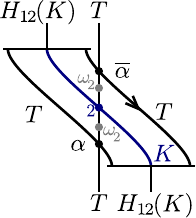} 
\ .
\end{equation}
For $H_{21}$ one proceeds analogously. Given $L\in\mcD_1$,  $H_{21}(L)$ has the same bimodule as $H_2(L)$, the $T$-crossings $\tau_2$ agree with those of $H_2$, while the $T$-crossings $\tau_1$ are given by
\begin{equation}
\tau_1^{H_{21}(L)} := \pic[1.5]{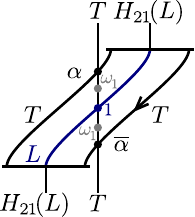}
\quad ,\qquad
\overline{\tau_1^{H_{21}(L)}} := \pic[1.5]{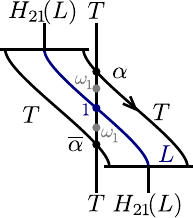} \ .
\end{equation}
One verifies that this makes $H_{12}(K)$, $H_{21}(L)$ into objects of $\mcC_\opA$.

On a morphism $f$ each functor acts as $\pi \circ (\id_T \otimes f \otimes \id_{T^*}) \circ \iota$, with $\pi$, $\iota$ the corresponding projection and embedding.
\begin{figure}
	\centering
	\begin{subfigure}[b]{0.4\textwidth}
		\centering
		\pic[1.5]{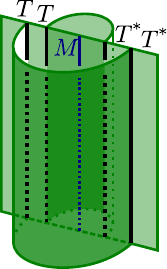}
		\caption{}
		\label{fig:H12_strat}
	\end{subfigure}
	\begin{subfigure}[b]{0.4\textwidth}
		\centering
		\pic[1.5]{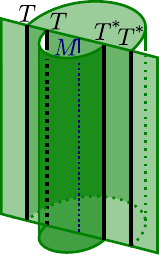}
		\caption{}
		\label{fig:H21_strat}
	\end{subfigure}
\caption{
Stratifications corresponding to the compositions (a) $H_{12}\circ H_2(M)$ and (b) $H_{21} \circ H_1(M)$.
}
\label{fig:H_strats}
\end{figure}

\begin{rem}
\label{rem:pipe_funct-strat}
		Let $M\in\mcACA$.
		As Wilson lines in the orbifold TQFT from Section~\ref{subsec:TQFT_motivation}, $H_{12}\circ H_2(M)$ and $H_{21} \circ H_1(M)$ correspond to the stratifications in Figure~\ref{fig:H_strats}.
		For this reason, we call these functors ``pipe functors''.
\end{rem}

As explained in Appendix \ref{App:sep_biadj}, for any pair of $\opk$-linear categories $\mcA$, $\mcB$, a biadjunction between functors $X:\mcA \ra \mcB$, $Y:\mcB \ra \mcA$ is called \it{separable}, if the natural transformation $\vareps\circ\widetilde{\eta}: \Id_\mcB \Ra \Id_\mcB$ is invertible (here $\widetilde{\eta}: \Id_\mcB \Ra XY$ is the unit of the adjunction $Y\dashv X$ and $\vareps:XY \Ra \Id_\mcB$ is the counit of the adjunction $X \dashv Y$).
Suppose now that the category $\mcA$ is finitely semisimple, $\mcB$ is idempotent complete, and that there exists a separable biadjunction between $\mcA$ and $\mcB$.
Then it is shown in Proposition \ref{thm:ssi_cond} that $\mcB$ is finitely semisimple as well.

Consider the following commuting square of forgetful functors:
\begin{equation}
\label{eq:U_functs}
\begin{tikzcd}
\mcACA
& \mcD_1\arrow[l, "U_1",swap]\\
\mcD_2\arrow[u, "U_2",swap]
& \mcC_\opA\arrow[l, "U_{12}",swap]\arrow[u, "U_{21}",swap]
\end{tikzcd}.
\end{equation}
We have:
\begin{prp}
\label{prp:biadjunctions}
$(H_1,U_1)$, $(H_{12}, U_{12})$, $(H_2,U_2)$, $(H_{21}, U_{21})$ are pairs of biadjoint functors and in each case the biadjunction is separable.
\end{prp}

\begin{proof}
We show this for $(H_1,U_1)$ only, the proofs for other cases are similar.
The biadjunction is given by maps
\begin{equation*}
\begin{array}{rcl}
 \mcD_1(H_1M,N) &\ra &\mcACA(M,U_1N)\\
 \varphi_{M,N}: \pic[1.5]{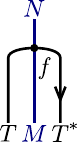} & \mapsto   & \pic[1.5]{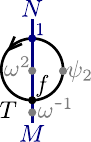} \cdot \phi\\
 \varphi^{-1}_{M,N}:\pic[1.5]{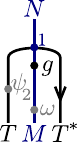} & \mapsfrom & \pic[1.5]{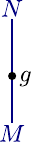}
\end{array}, \quad
\begin{array}{rcl}
 \mcD_1(K,H_1L) &\ra &\mcACA(U_1K,L)\\
 \chi_{K,L}: \pic[1.5]{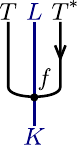} & \mapsto   & \pic[1.5]{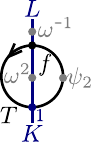} \cdot \phi\\
 \chi^{-1}_{K,L}:\pic[1.5]{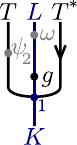} & \mapsfrom & \pic[1.5]{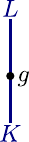}
\end{array}
\end{equation*}
for all $M,L\in\mcACA$, $N,K\in\mcD_1$.
As an example, we will check that $\varphi_{M,N}$ and $\varphi_{M,N}^{-1}$ are indeed inverses of each other: Let $f\in\mcD_1(H_1M, N)$, $g\in\mcACA(M,U_1N)$.
Then $\varphi_{M,N}\circ\varphi_{M,N}^{-1}(g) = g$ follows by applying
\eqrefT{5} and \eqrefO{8} to remove the $T$-loop
while the other composition needs the following more elaborate computation:
\begin{align*}
& \varphi_{M,N}^{-1}\circ\varphi_{M,N}(f) = 
 \pic[1.5]{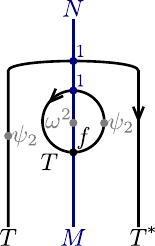}
 \cdot \phi
   \stackrel{\eqrefO{4}}{=}  
   \pic[1.5]{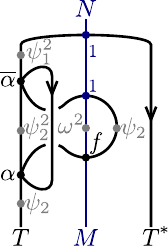}
   \cdot \phi
\\
&
\underset{\eqrefT{4}}{\stackrel{\eqref{eq:T1bp}}{=}} 
\pic[1.5]{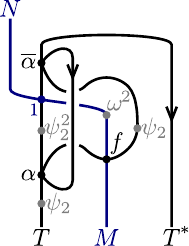}
\cdot \phi
\overset{\text{deform}}=
\pic[1.5]{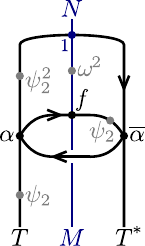}
\cdot\phi
\stackrel{\eqref{eq:HM_Tcross_inv}}{=}
 \pic[1.5]{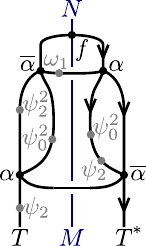}
 \cdot\phi
\\
&
 \stackrel{\eqrefO{3}}{=}
 \pic[1.5]{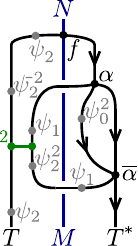}
 \cdot\phi
  \stackrel{\eqrefO{3}}{=}
 \pic[1.5]{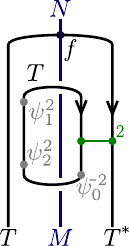}
 \cdot\phi
 \stackrel{\eqrefO{8}}{=} f \ .
\end{align*}
Similarly, $\chi_{K,L}$ and $\chi_{K,L}^{-1}$ are inverses of each other.
One also checks that $\varphi^{-1}_{M,N}(g)$, $\chi^{-1}_{K,L}(g)$ are morphisms in $\mcD_1$.

As always, the counit $\vareps: H_1 U_1 \Ra \Id_{\mcD_1}$ is given by $\{\vareps_N = \varphi^{-1}_{U_1N,N}(\id_{U_1N})\}_{N\in\mcD_1}$ and the unit $\widetilde{\eta}: \Id_{\mcD_1} \Ra H_1 U_1$ by $\{\widetilde{\eta_N} = \chi^{-1}_{N,U_1N} (\id_{U_1N})\}_{N\in\mcD_1}$.
For all $N\in\mcD_1$ one has:
\begin{equation*}
\vareps_N \circ \widetilde{\eta_N}
= \pic[1.5]{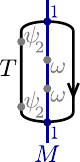} 
\stackrel{\eqrefT{4}}{=} \pic[1.5]{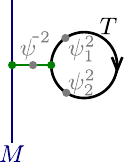} 
\stackrel{\eqrefO{8}}{=} \pic[1.5]{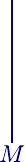} \cdot \phi^{-1},
\end{equation*}
i.e.\ $\vareps_N \circ \widetilde{\eta_N}$ is invertible and hence the biadjunction is separable.
\end{proof}

\begin{rem}
\label{rem:pipe_funct}
Note that the diagram in~\eqref{eq:U_functs} commutes with identity natural isomorphism, $U := U_1 \circ U_{21} = U_2 \circ U_{12} : \mcC_\opA \to \mcACA$, as each path sends an object in $\mcC_\opA$ to its underlying bimodule. By Proposition~\ref{prp:biadjunctions}, both $H_{21} \circ H_1$ and $H_{12} \circ H_2$ are biadjoint to $U$, and hence in particular naturally isomorphic. Thus the diagram in \eqref{eq:H_functs} commutes as well.
In view of the stratifications in Figure~\ref{fig:H_strats} this is not surprising, and an explicit natural isomorphism can be build from $\a$ and $\abar$.
Below we will work exclusively with the the composition
\begin{equation}
	P := H_{12} \circ H_2 : \mcACA \to \mcC_\opA  \ ,
\end{equation}
where `$P$' stands for `pipe functor'.
\end{rem}

\begin{prp}\label{prp:CA-fin-ssi}
The categories $\mcD_1$, $\mcD_2$ and $\mcC_\opA$ are finitely semisimple.
\end{prp}
\begin{proof}
From Proposition \ref{prp:ACA_is_ssi} we already know that $\mcACA$ is finitely semisimple.
Therefore by Proposition \ref{prp:biadjunctions} and the
argument preceding it, it is enough to show that $\mcD_1$, $\mcD_2$ and $\mcC_\opA$ are idempotent complete.
We show this for $\mcD_1$ only, since the other cases are analogous.
Let $p:M \ra M$ be an idempotent in $\mcD_1$.
Then it is also an idempotent in $\mcACA$ and hence has a retract $(S,e,r)$ in $\mcACA$.
Equip $S$ with the morphisms:
\begin{equation}
\tau_1^S := \pic[1.5]{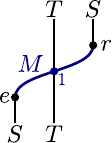}, \qquad
\overline{\tau_1^S} := \pic[1.5]{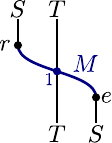}.
\end{equation}
They satisfy the axioms of $T$-crossings, e.g.
\begin{align*}
\pic[1.5]{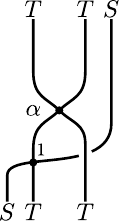} &= \pic[1.5]{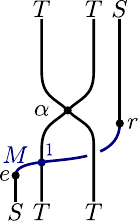} = \pic[1.5]{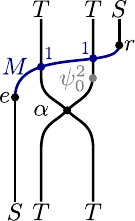} = \pic[1.5]{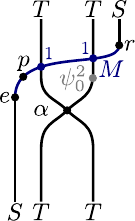} \overset{(*)}= \pic[1.5]{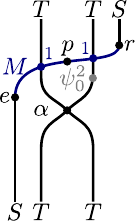}\\
&= \pic[1.5]{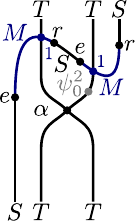} = \pic[1.5]{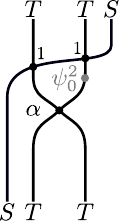} \ ,
\end{align*}
where in step $(*)$ we used that $p$ is a morphism in $\mcD_1$.
The argument that $e$ and $r$ are morphisms in $\mcD_1$ is similar.
$(S,e,r)$ is therefore a retract in $\mcD_1$.
\end{proof}

Combining Propositions~\ref{prp:D_is_spherical} and~\ref{prp:CA-fin-ssi} with Remark~\ref{rem:mutifusion-to-ribbon}, we get:

\begin{cor}\label{cor:CA-multifusion-ribbon}
$\mcC_\opA$ is a ribbon multifusion  category.
\end{cor}

\subsection{Modularity}
In this section we will in addition assume that $\opA$ is a simple orbifold datum (see Definition \ref{def:simple_orb}), so that by Corollary~\ref{cor:CA-multifusion-ribbon}, $\mcC_\opA$ is a ribbon fusion category. We will show that $\mcC_\opA$ is in fact modular.

Let $\Ind_A: \mcC \ra \mcACA$, $X \mapsto A \otimes X \otimes A$ be the induced bimodule functor.
It is biadjoint to the forgetful functor $U_{AA}: \mcACA \ra \mcC$ (e.g.\ apply the adjunction for left modules in \cite[Prop.\ 4.10, 4.11]{FS} to the algebra $A \otimes A^\opp$).

We will use the pipe functor $P = H_{12} \circ H_2: \mcACA \ra \mcC_\opA$, which is biadjoint to the forgetful functor $U : \mcC_\opA \ra \mcACA$, cf.\ Remark~\ref{rem:pipe_funct}.
It will prove useful to note that e.g. for $M\in\mcACA$, the braiding of $P(M)$ with any object $N\in\mcC_\opA$ can be simplified as follows:
\begin{equation}
\label{eq:braidings_w_FM}
c_{PM,N} = \pic[1.5]{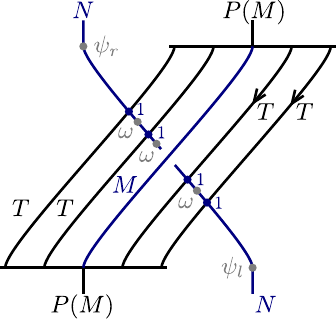}, \quad
c^{-1}_{N,PM} = \pic[1.5]{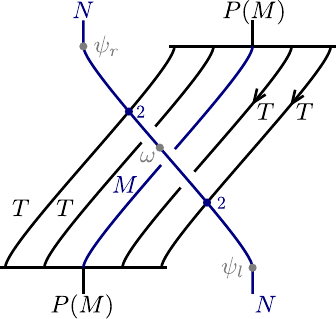}.
\end{equation}

Let $M, N \in \mcC_\opA$ and $f:M \ra N$ be a morphism in $\mcACA$.
Define the averaged morphism $\overline{f}: M \ra N$ to be the $A$-$A$-bimodule morphism
\begin{equation}
\label{eq:average}
\overline{f} := \pic[1.5]{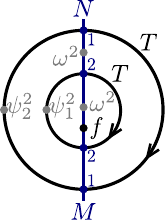} \cdot \phi^2.
\end{equation}
One can check that $\overline{f}$ is a morphism in $\mcC_\opA$.
Moreover, if $f\in\mcC_\opA$, then $f = \overline{f}$, 
i.e.\ averaging is an idempotent on the morphism spaces of $\mcACA$, projecting onto the morphism spaces of $\mcC_\opA$.

Recall the notations $|X|_\mcC$, $\tr_\mcC f$, $\Irr_\mcC$ and $\Dim \mcC$ from Section \ref{subsec:conventions}.
For the remainder of the section, a thick loop (red in the online version) in a string diagram $D$ will mean then sum $\sum_{k\in\Irr_\mcC}|k|_\mcC D_k$, where $D_k$ denotes the string diagram in which 
the red loop is labelled by $k\in\Irr_\mcC$.
\begin{lem}
\label{prp:modularity_cond}
(cf. \cite{KO}, Lemma 4.6)
A premodular category $\mathcal{E}$ is modular iff for all $i\in\Irr_{\mathcal{E}}$
\begin{equation}
\label{eq:modularity_cond}
\pic[1.5]{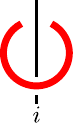}
=
c \cdot \d_{i,\opid} \cdot \pic[1.5]{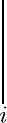}
\end{equation}
for some $c\neq 0$. Moreover, in this case one necessarily has $c = \Dim\mathcal{E}$.
\end{lem}

A useful corollary of Lemma \ref{prp:modularity_cond} is the following identity, which holds for any modular fusion category $\mcC$ and an object $X\in\mcC$:
\begin{equation}
\label{eq:id_for_mod_cats}
\pic[1.5]{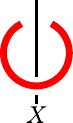} = \Dim \mcC \cdot \sum_{\a} \pic[1.5]{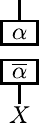},
\end{equation}
where $\a$ and $\abar$ run over a basis of $\mcC(\opid,X)$ and its dual (with respect to the composition pairing $\mcC(\opid,X)\otimes_\opk\mcC(X,\opid) \ra \mcC(\opid,\opid)\cong\opk$).
\begin{lem}
\label{lem:tr_inD}
Let $\opA$ be a simple orbifold datum in $\mcC$.
For $M \in \mcC_\opA$ and $f \in \End_{\mcC_\opA}(M)$ one has:
\begin{equation}\label{trace-formula-CA}
\tr_{\mcC_\opA} f \cdot \tr_\mcC \psi^4 = \tr_\mcC (\omega_M^2 \circ f).
\end{equation}
In particular, $|M|_{\mcC_\opA} \cdot \tr_\mcC \psi^4 = \tr_\mcC \omega_M^2$.
\end{lem}
\begin{proof}
From expressions \eqref{eq:evs_evts_in_D} one gets:
\begin{equation*}
\pic[1.5]{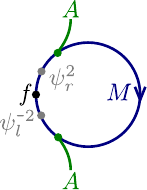} 
\overset{(*)}= 
\tr_{\mcC_\opA}f \cdot \pic[1.5]{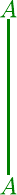}
\quad\Ra\quad
\pic[1.5]{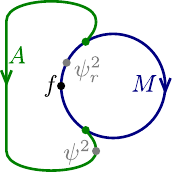} = \tr_{\mcC_\opA}f \pic[1.5]{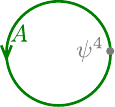}
~,
\end{equation*}
where in $(*)$ we used that $A$ is simple in $\mcC_\opA$.	
\end{proof}
\begin{thm}
\label{thm:CA_is_modular}
Let $\opA$ be a simple orbifold datum in $\mcC$.
Then
\begin{enumerate}[i)]
\item
$\mcC_\opA$ is a modular fusion category;
\item
$\tr_\mcC \psi^4 \neq 0$ and $\Dim \mcC_\opA = \frac{\Dim\mcC}{\phi^4 \cdot (\tr_\mcC\psi^4)^2}$.
\end{enumerate}
\end{thm}
\begin{proof}
Let $i\in\mcC$, $\mu\in\mcACA$, $\D\in\mcC_\opA$ be simple objects. 
One has the following decompositions (the label over the isomorphism sign indicates the category it holds in)
\begin{align}
&\mathrm{a)}~~ \mu \overset{\mcC}{\cong} \bigoplus_{k\in\Irr_\mcC} k \otimes \mcC(k,\mu) \ ,
&&\mathrm{b)}~~ \D \overset{\mcACA}{\cong} \bigoplus_{\nu\in\Irr_{\mcACA}} \nu \otimes \mcACA(\nu,\D) \ , \nonumber\\
&\mathrm{c)}~~AiA \overset{\mcACA}{\cong} \bigoplus_{\nu\in\Irr_{\mcACA}} \nu \otimes \mcC(i,\nu) \ ,
&&\mathrm{d)}~~ P(\mu) \overset{\mcC_\opA}{\cong} \bigoplus_{\Lambda\in\Irr_{\mcC_\opA}} \Lambda \otimes \mcACA(\mu, \Lambda) \ ,
\label{main-proof-object-decomp}
\end{align}
Here, the forgetful functors $U : \mcC_\opA \to \mcACA$ and $U_{AA} : \mcACA \to \mcC$ are not written out.
The isomorphisms in the second row follow from the biadjunctions $\Ind_A\dashv U_{AA}$ and $P \dashv U$, respectively.
For a simple $\mu\in\mcACA$ and $f\in\End_{AA}(\mu)$, let $\langle f \rangle\in\opk$ be such that $f = \langle f \rangle \cdot \id_\mu$.
For a fixed simple $\D\in\mcC_\opA$, let $L_\D \in \End_{\mcC_\opA}(\D)$ be the morphism as on the left hand side of \eqref{eq:modularity_cond} (now understood as a string diagram in $\mcC_\opA$).
Use Lemma \ref{lem:tr_inD} and the decompositions above to obtain the equalities (in this computation, all string diagrams are written in $\mcC_\opA$)
\begin{align}
&L_\D \cdot \tr_\mcC \psi^4
= \sum_{\Lambda\in\Irr_{\mcC_\opA}} |\Lambda|_{\mcC_\opA} \cdot \tr_\mcC \psi^4                \pic[1.5]{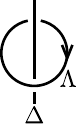}
 \overset{\eqref{trace-formula-CA}}= \sum_{\Lambda\in\Irr_{\mcC_\opA}} \tr_\mcC \omega^2_\Lambda                                  \pic[1.5]{34_mod_proof_01.pdf}
 \nonumber\\
&\overset{\text{(\ref{main-proof-object-decomp}\,b)}}= \sum_{\Lambda,\nu} \tr_\mcC \omega^2_\nu \cdot \dim\mcACA(\nu,\Lambda)                       \pic[1.5]{34_mod_proof_01.pdf}
 \overset{\text{(\ref{main-proof-object-decomp}\,d)}}= \sum_{\nu\in\Irr_{\mcACA}} \tr_\mcC \omega^2_\nu                                             \pic[1.5]{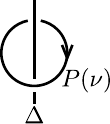}
  \nonumber\\
&= \sum_{\nu\in\Irr_{\mcACA}} |\nu|_\mcC \cdot \langle\omega^2_\nu\rangle                                   \pic[1.5]{34_mod_proof_02.pdf}
 = \sum_{\nu\in\Irr_{\mcACA}} |\nu|_\mcC                                                        \pic[1.5]{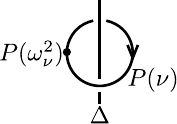}
  \nonumber\\
&\overset{\text{(\ref{main-proof-object-decomp}\,a)}}= \sum_{\nu,k} |k|_\mcC \cdot \dim\mcC(k,\nu)                                                  \pic[1.5]{34_mod_proof_03.pdf}
 \overset{\text{(\ref{main-proof-object-decomp}\,c)}}= \sum_{k\in\Irr_\mcC} |k|_\mcC                                                                \pic[1.5]{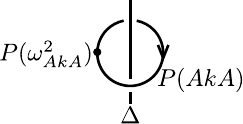} \ .
 \label{eq:LDelta-psi4}
\end{align}
Next, use the expressions \eqref{eq:evs_evts_in_D} for (co-)evaluation maps and \eqref{eq:braidings_w_FM} for braidings to compute (in the following the string diagrams are again in $\mcC$):
\begin{align}
&\tr_\mcC \big(\omega^2_\D \circ L_\D \cdot \tr_\mcC \psi^4 \big)
 \overset{\eqref{eq:LDelta-psi4}}= \sum_{k\in\Irr_\mcC} |k|_\mcC \pic[1.5]{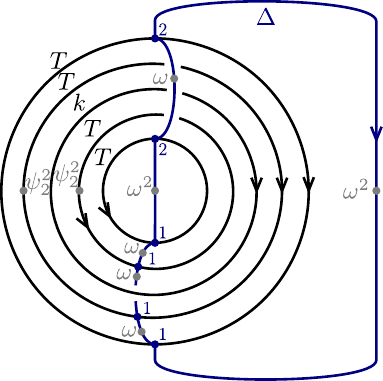}
\nonumber \\
&
\overset{\text{deform}}= 
\pic[1.5]{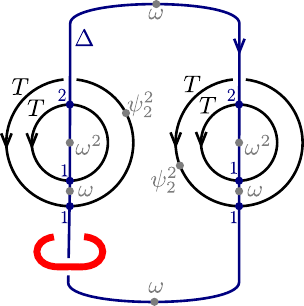}
\stackrel{\eqref{eq:id_for_mod_cats}}{=} 
 \Dim\mcC\sum_\a \pic[1.5]{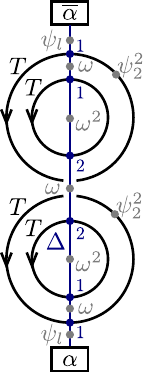}
\nonumber \\
&\stackrel{(*)}{=}
  \frac{\Dim\mcC}{\phi^2} \sum_\a \pic[1.5]{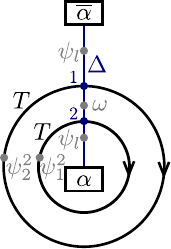}
\overset{(**)}=
\frac{\Dim\mcC}{\phi^2} \sum_\a \pic[1.5]{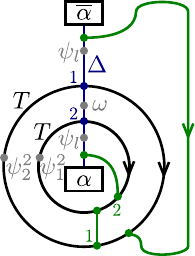}
\nonumber \\
&\stackrel{\eqref{eq:unit_T-cross}}{=} 
 \frac{\Dim\mcC}{\phi^2} \sum_\a \pic[1.5]{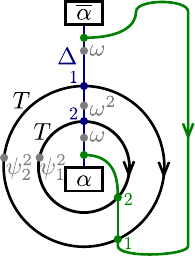}
 = \frac{\Dim\mcC}{\phi^4} \sum_\a \phi^2 \pic[1.5]{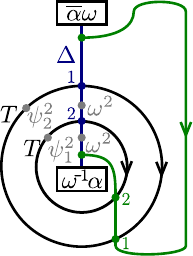}
 ~.
\label{eq:modular-aux-1} 
\end{align}
Step $(**)$ is best checked in reverse:
the $A$-strings can be removed using the intertwining properties of $\tau_i$ and $\Delta$-separability.
Step $(*)$ consists of two computations, each of which combines two of the four $T$-loops into one. We will only show the first:
\begin{align*}
										   \pic[1.5]{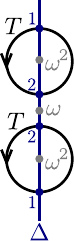} 
&\stackrel{\eqrefO{5}}{=}                  \pic[1.5]{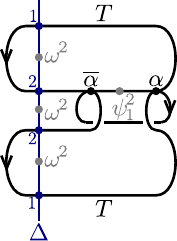}
 \underset{\eqrefT{6}}{\stackrel{\eqref{eq:T3bp}}{=}}
        \pic[1.5]{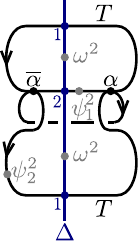}
\overset{\text{deform}}=
\pic[1.5]{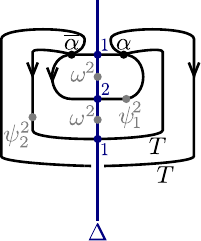}\\ 
&
\underset{\eqrefT{5}}{\stackrel{\eqrefT{2}}{=}}
\pic[1.5]{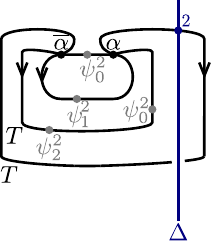}
 \stackrel{\eqrefO{7}}{=} 			       \pic[1.5]{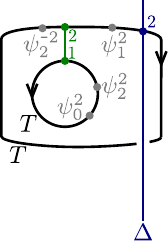}
 \stackrel{\eqrefO{8}}{=}                  \pic[1.5]{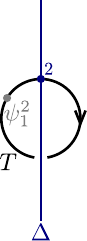} \cdot \phi^{-1}.
\end{align*}
The last term in \eqref{eq:modular-aux-1} contains the average of a morphism as defined in~\eqref{eq:average} which projects onto $\mcC_\opA(A,\D)$. The sum over $\alpha$ therefore computes the trace of this projection and one has
\begin{equation}\label{eq:CA-modular-aux2}
\tr_\mcC \big(\omega^2_\D \circ L_\D \cdot \tr_\mcC \psi^4 \big) =
\frac{\Dim \mcC}{\phi^4} \cdot \dim\mcC_\opA(A,\D) =
\frac{\Dim \mcC}{\phi^4} \cdot \d_{A,\D}
~.
\end{equation}
It follows that $\tr_\mcC \psi^4 \neq 0$, as the right hand side is non-zero for $A=\D$.
This proves the first claim in part (ii) of the theorem.

Recall from Section~\ref{subsec:conventions} that since $\mcC_\opA$ is fusion, $|\D|_{\mcC_\opA} \neq 0$ for all simple $\D$.
Using this, we finally get
\begin{align}
L_\D &= 
\frac{\tr_{\mcC_\opA} L_\D \cdot \tr_\mcC \psi^4}{|\D|_{\mcC_\opA} \cdot \tr_\mcC \psi^4} \cdot \id_\D
\overset{\eqref{trace-formula-CA}}= 
\frac{\tr_\mcC (\omega^2_\D \circ L_\D)}{|\D|_{\mcC_\opA} \cdot \tr_\mcC \psi^4} \cdot \id_\D 
\\
&\overset{\eqref{eq:CA-modular-aux2}}=
\frac{\Dim \mcC}{|\D|_{\mcC_\opA} \cdot \phi^4 \cdot (\tr_\mcC \psi^4)^2}\cdot \d_{A,\D} \cdot \id_\D 
\ .
\end{align}
Lemma~\ref{prp:modularity_cond} now implies part (i) and the remaining claim in part (ii).
\end{proof}
\section{Examples} \label{sec:examples}
In this section we look into examples \ref{eg1} and \ref{eg2} in the introduction, that is, the cases of an orbifold datum obtained from a commutative simple $\D$-separable Frobenius algebra, and from a spherical fusion category.
\subsection{Local modules}\label{sec:locmod}
Let $A\in\mcC$ be a commutative $\D$-separable Frobenius algebra.
We call an $A$-module $M$ \it{local} (or \it{dyslectic}) if
\begin{equation}
\label{eq:loc_cond}
\pic[1.5]{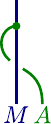} = \pic[1.5]{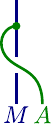}.
\end{equation}
The category of local modules will be denoted by $\mcC_A^\loc$ (see \cite{FFRS} for more details and further references on local modules).

Note that since $A$ is commutative, for any $A$-module $M$ the morphisms on both sides of \eqref{eq:loc_cond} define right $A$-actions on $M$, which yields two bimodules $M_+$ and $M_-$.
Local modules are precisely those for which one has $M_+ = M_-$.
One uses the tensor product of bimodules to equip $\mcC_A^\loc$ with tensor product and duals.
Furthermore, $\mcC$ induces the braiding and the twists on $\mcC_A^\loc$.
It was proven in \cite{KO} that if $A$ is \it{haploid}
(i.e.\ $\dim \mcC(\opid, A) = 1$, cf.~\cite{FS}), 
then $\mcC_A^\loc$ is in fact a modular fusion category.

For the remainder of the section, let $A$ be a haploid $\D$-separable commutative Frobenius algebra in a modular fusion category $\mcC$.
Then it is automatically symmetric (see \cite[Cor.\,3.10]{FRS1}), and as shown in \cite[Sec.\,3.4]{CRS3}, it gives an orbifold datum
\begin{equation}
\label{eq:loc_mods_orb_data}
\opA ~=~ \big(\, A
\,,\, 
T = A
\,,\,  
\a = \abar = \D\circ\mu
\,,\, 
\psi = \id_A
\,,\,  
\phi = 1
\,\big) \ ,
\end{equation}
where one uses commutativity of $A = T$ to treat it as $A$-$AA$-bimodule.
The rest of the section is dedicated to proving the following
\begin{thm}
\label{thm:CA_equiv_CAloc}
Let $A$ be a haploid $\D$-separable
commutative Frobenius algebra in $\mcC$, and let 
$\opA$ be the orbifold datum in \eqref{eq:loc_mods_orb_data}.
Then  $\opA$ is simple and
one has $\mcC_\opA \cong \mcC_A^\loc$ as $\opk$-linear ribbon categories.
\end{thm}
\begin{proof}
Define a functor $F:\mcC_A^\loc \ra \mcC_\opA$ as follows:
Given a local module $M$, equip it with the canonical bimodule structure and define the $T$-crossings to be
\begin{equation}
\label{eq:Aloc_T_cross}
\pic[1.5]{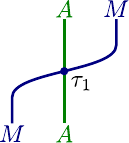} = \pic[1.5]{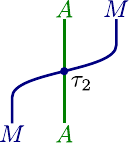} := \pic[1.5]{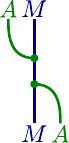}.
\end{equation}
All of the axioms then hold and are easy to check, e.g.
\begin{align*}
\pic[1.5]{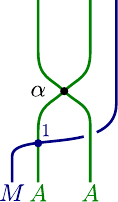}
= \pic[1.5]{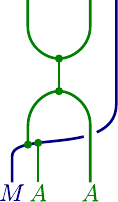}
= \pic[1.5]{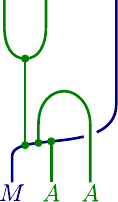}
\overset{(*)}= \pic[1.5]{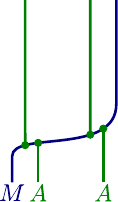}
= \pic[1.5]{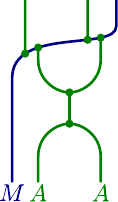}
= \pic[1.5]{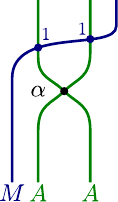} \ .
\end{align*}
In $(*)$ one uses the fact that the right action of $M$ comes from \eqref{eq:loc_cond}.
A morphism in $\mcC_\opA$ is precisely an $A$-module morphism, i.e.\ $F$ is fully faithful.
Since $A$ is simple as a left module over itself (because $A$ is haploid), this shows that the orbifold datum $\opA$ is simple.

It is easy to check that $F$ preserves tensor products, braidings and twists, hence it only remains to check that it is an equivalence.
We show that $F$ is essentially surjective.

Let $(M,\tau_1,\tau_2, \overline{\tau_1}, \overline{\tau_2}) \in \mcC_\opA$.
Since $\tau_1$, $\tau_2$ are $A$-$AA$-bimodule morphisms, one has:
\begin{equation}
\label{eq:Aloc-A-AA-bimod_conds}
\pic[1.5]{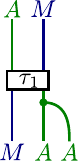} = \pic[1.5]{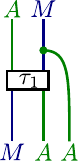} = \pic[1.5]{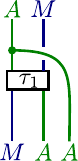}, \qquad
\pic[1.5]{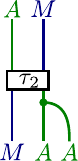} = \pic[1.5]{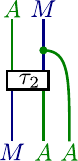} = \pic[1.5]{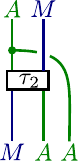}.
\end{equation}
For example in the first equality for $\tau_1$ we think of the right $A$ action as the action of the first tensor factor of $A \otimes A$ and in the second equality as the action of the second tensor factor.
Since $M \cong A \otimes_A M \cong M \otimes_A A$, the $T$-crossings $\tau_1$, $\tau_2$ can be recovered from the following invertible $A$-module morphisms $\widehat{\tau_1}, \widehat{\tau_2}: M \ra M$:
\begin{equation}
\widehat{\tau_i} := \pic[1.5]{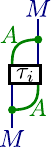}, \qquad i=1,2.
\end{equation}
We can then relate the left and right action on $M$ as follows:
\begin{equation*}
\pic[1.5]{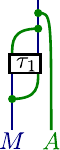} 
\overset{\eqref{eq:Aloc-A-AA-bimod_conds}}= 
\pic[1.5]{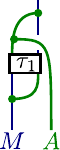}  = \pic[1.5]{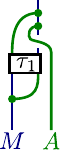}  = \pic[1.5]{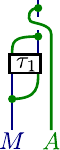} 
\quad\Ra\quad
\pic[1.5]{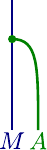}  = \pic[1.5]{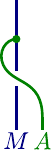} .
\end{equation*}
Similarly, the identities for $\tau_2$ in \eqref{eq:Aloc-A-AA-bimod_conds} imply
\begin{equation*}
\pic[1.5]{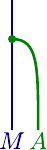} = \pic[1.5]{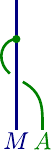}.
\end{equation*}
Hence $M$ is a local module with the canonical bimodule structure.
It remains to show that the $T$-crossings are as in \eqref{eq:Aloc_T_cross}.
Using the identity \eqrefT{1} one has
\begin{equation*}
\pic[1.5]{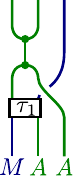} = \pic[1.5]{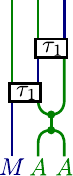} \quad\Lra\quad
\pic[1.5]{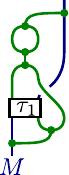} = \pic[1.5]{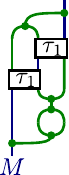}.
\end{equation*}

Examining both sides of the last equality gives:
\begin{align*}
&\text{left hand side: }
\pic[1.5]{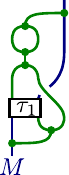} = \pic[1.5]{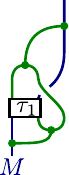} = \pic[1.5]{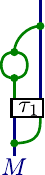} = \widehat{\tau_1},\\
&\text{right hand side: }
\pic[1.5]{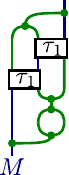} = \pic[1.5]{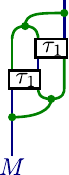} = \pic[1.5]{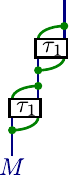} = \widehat{\tau_1} \circ \widehat{\tau_1}.
\end{align*}
Hence one has $\widehat{\tau_1} = \widehat{\tau_1} \circ \widehat{\tau_1}$ and since it is invertible, $\widehat{\tau_1} = \id_M$, which in turn implies that $\tau_1$ is precisely as in \eqref{eq:Aloc_T_cross}.
The identity \eqrefT{3} implies the same for $\tau_2$.
\end{proof}
Combining the above result with Theorem~\ref{thm:CA_is_modular} gives an independent proof that $\mcC_A^\loc$ is modular.
For the orbifold datum \eqref{eq:loc_mods_orb_data} one has $\tr_\mcC \psi^4 = |A|_\mcC$ and $\phi = 1$, so that the second part of Theorem \ref{thm:CA_is_modular} yields
\begin{equation}
\Dim \mcC_A^\loc = \frac{\Dim \mcC}{|A|_\mcC^2} \ .
\end{equation}
Both, modularity and the above dimension formula are already known from \cite{KO}.
\subsection{Drinfeld centre}\label{Drinfeld-example}
In this section, fix a spherical fusion category $\mcS$ with $\Dim\mcS\neq 0$ (this condition is only relevant if $\operatorname{char} \opk \neq 0$ see Section~\ref{subsec:conventions}).
We will not assume $\mcS$ to be strict; its associator and unitors will be denoted by $a_{X,Y,Z}: (XY)Z \ra X(YZ)$, $l_X:\opid X \ra X$, $r_X: X\opid \ra X$, for all $X,Y,Z\in\mcS$.

Recall that the Drinfeld centre $\mcZ(\mcS)$ consists of pairs $(X,\g)$, where $X\in\mcS$ and $\g: X \otimes - \Ra -\otimes X$ is a natural transformation, satisfying the hexagon identity,
i.e.\ $(\id_U \otimes \g_V)\circ a_{UXV}\circ(\g_U \otimes \id_V) = a_{U,V,X}\circ \g_{UV}\circ a_{X,U,V}$ for all $U,V \in \mcS$.
A morphism $f:(X,\g) \ra (Y,\d)$ is a morphism $f:X\ra Y$ in $\mcS$, such that $(\id_U \otimes f)\circ\g_U = \d_U\circ(f\otimes \id_U)$ for all $U\in\mcS$.

$\mcZ(\mcS)$ is a ribbon category with monoidal product
\begin{equation}
(X,\g) \otimes (Y,\d) ~:=~ ( X \otimes Y, ~\Gamma^{XY} ).
\end{equation}
where, for all $U\in\mcS$,
\begin{equation}
\Gamma^{XY}_U :=
\left[
\begin{array}{rl}
(XY)U &\xra{a_{X,Y,U}} X(YU) \xra{\id_X\otimes\d} X(UY) \xra{a^{-1}_{X,U,Y}}(XU)Y\\
	  &\xra{\g_U \otimes \id_Y} (UX)Y \xra{a_{U,X,Y}} U(XY)
\end{array}
\right] \ .
\end{equation}
The braiding and the twist are given by
\begin{align}
\label{eq:ZS_br_tw}
&\big[c_{(X,\g),(Y,\d)}:(X\otimes Y, \Gamma^{XY}) \ra (Y\otimes X, \Gamma^{YX})\big]
:= \big[ X\otimes Y \xra{\g_Y} Y \otimes X  \big] ,\\ \nonumber
& \theta_{(X,\g)}
:= \left[ 
\begin{array}{l}
X \xra{r_X^{-1}} X\opid \xra{\id_X\otimes \coev_X} X(XX^*) \xra{a_{X,X,X^*}^{-1}} (XX)X^*\\
~~~\xra{\g_X\otimes \id_{X^*}} (XX)X^* \xra{a_{X,X,X^*}} X(XX^*)  \xra{\id_X \otimes \evt_X} X\opid \xra{r_X} X
\end{array}
\right].
\end{align}

Let us recall from \cite[Sec.\,4]{CRS3} how one can associate to $\mcS$ an orbifold datum $\opA^\mcS$ in the trivial modular fusion category $\Vect_\opk$ of finite dimensional $\opk$-vector spaces.\footnote{
	Our conventions here differ from those in \cite{CRS3}. For example, instead of $T = \bigoplus_{i,j,l} \mcS(l,ij)$ as in \eqref{eq:sph_cat_orb_data}, in \cite{CRS3} the bimodule $\bigoplus_{i,j,l} \mcS(ij,l)$ is used.
	The convention used here is better suited for the equivalence proof.
}
For brevity, denote $\mcI := \Irr_\mcS$ and $|X| := |X|_\mcS$
(recall the conventions in Section~\ref{subsec:conventions}).
For each $i\in\mcI$, fix square roots $|i|^{1/2}$ and define the natural transformation $\psi:\Id_\mcS \Ra \Id_\mcS$ by taking for each $X\in\mcS$
\begin{equation}
\label{eq:psi_nat_transf}
\psi_X := \sum_{i,\pi} |i|^{1/2} \, [X \xra{\pi} i \xra{\pibar} X] \ .
\end{equation}
Here, $i$ in the sum ranges over $\mcI$, $\pi$ over a basis of $\mcS(X,i)$, and
$\pibar$ is the corresponding element of the dual basis of $\mcS(i,X)$ with 
respect to the composition pairing $\mcS(i,X) \otimes_\opk \mcS(X,i) \ra \mcS(i,i) \cong \opk$.
Now define $\opA^\mcS = (A, T, \a, \abar, \psi, \phi)$ with:
\begin{align}
& A = \bigoplus_{i\in\mcI} \mcS(i,i) \cong \bigoplus_{i\in\mcI} \opk, \qquad
T = \bigoplus_{l,i,j\in\mcI} \mcS(l,ij),\nonumber \\ 
&\begin{array}{rll}
\a: \underbrace{\bigoplus_{l,a,i,j,k} \mcS(l,ia) \otimes_\opk \mcS(a,jk)}_{\cong \bigoplus_{l,i,j,k} \mcS(l,~i(jk))} &\longrightarrow
&\underbrace{\bigoplus_{l,b,i,j,k} \mcS(l,bk) \otimes_\opk \mcS(b,ij)}_{\cong \bigoplus_{l,i,j,k} \mcS(l,~(ij)k)},\\
\left[ l \xra{f} i(jk)\right] &\longmapsto & \left[l \xra{f} i(jk) \xra{a^{-1}_{i,j,k}} (ij)k \xra{\psi_{ij}^{-2} \otimes \id_k} (ij)k \right],
\end{array}
\nonumber \\
&\begin{array}{rll}
\abar: \underbrace{\bigoplus_{l,b,i,j,k} \mcS(l,bk) \otimes_\opk \mcS(b,ij)}_{\cong \bigoplus_{l,i,j,k} \mcS(l,~(ij)k)} &\longrightarrow
&\underbrace{\bigoplus_{l,a,i,j,k} \mcS(l,ia) \otimes_\opk \mcS(a,jk)}_{\cong \bigoplus_{l,i,j,k} \mcS(l,~i(jk))},\\
\left[ l \xra{g} (ij)k \right] &\longmapsto & \left[l \xra{g} (ij)k \xra{a_{i,j,k}} i(jk) \xra{\id_i \otimes \psi_{jk}^{-2}} i(jk) \right],
\end{array}
\nonumber
\\
&\psi: \left[i \xra{f} i\right] \,\longmapsto\, \left[i \xra{f} i \xra{\psi_i} i\right]
~~, \quad
\phi = \frac{1}{\Dim \mcS} = \left( \sum_{i\in\mcI} |i|^2 \right)^{-1}
\ .
\label{eq:sph_cat_orb_data}
\end{align}
Here, we abuse notation by denoting the morphism $\psi:A\ra A$ in the orbifold datum and the natural transformation $\psi:\Id_\mcS \Ra \id_\mcS$ from \eqref{eq:psi_nat_transf} with the same symbol.
The left action of $[f:k\ra k] \in A$ on $[m:l\ra ij]\in T$, $i,j,k,l \in \mcI$ is precomposition and the first (resp.\ second) right action is postcomposition with $(f \otimes \id_j)$ (resp.\ $(\id_i \otimes f)$) (if the composition is undefined, the corresponding action is by $0$).
The isomorphisms in the definitions of $\a$, $\abar$ given by composition.
For example, in the source object of $\alpha$, the explicit form of the isomorphism is $f \otimes_\opk g \mapsto (\id_i \otimes g) \circ f$.

Our goal is to prove the following
\begin{thm}
\label{thm:D_equiv_centre}
Let $\mcS$ be a spherical fusion category, $\opA^\mcS$ the orbifold datum as in \eqref{eq:sph_cat_orb_data} and $\mcC = \Vect_\opk$.
Then $\opA^\mcS$ is simple and $\mcC_{\opA^\mcS} \cong \mcZ(\mcS)$ as $\opk$-linear ribbon categories.
\end{thm}
Together with Theorem~\ref{thm:CA_is_modular} this gives an alternative proof that $\mcZ(\mcS)$ is modular. Furthermore, 
for the orbifold datum \eqref{eq:sph_cat_orb_data} one has $\tr_\mcC \psi^4 = \Dim\mcS$, so the second part of Theorem~\ref{thm:CA_is_modular} yields
\begin{equation}
\Dim \mcZ(\mcS) = (\Dim \mcS)^2 \ .
\end{equation}
Modularity and the dimension of $\mcZ(\mcS)$ are of course already known from \cite{Mu}.

\medskip

The proof of Theorem \ref{thm:D_equiv_centre} is somewhat lengthy and technical and is organised as follows: 
In Section \ref{subsubsec:aux_cat_AS} we define an auxiliary category $\mcA(\mcS)$ which is proved to be equivalent to the centre $\mcZ(\mcS)$ as a linear category.
Then in Sections \ref{subsubsec:functorD} and \ref{subsubsec:taus-satisfy-Tcond} we show that $\mcC_{\opA^\mcS} \cong \mcA(\mcS)$ as linear categories, and that the orbifold datum $\opA^\mcS$ is simple. Composing the two equivalences gives a linear equivalence $F:\mcZ(\mcS) \ra \mcC_{\opA^\mcS}$.
In Section \ref{subsubsec:F-ribb} we equip $F$ with a monoidal structure  and show that it preserves braidings and twists.
\subsubsection{Auxiliary category $\boldsymbol{\mcA(\mcS)}$ and equivalence to $\boldsymbol{\mcZ(\mcS)}$}
\label{subsubsec:aux_cat_AS}
\begin{defn}
Define the category $\mcA(\mcS)$ to have
\begin{itemize}[wide, labelwidth=!, labelindent=0pt]
\item \it{objects}: triples $(X, t^X, b^X)$, where $X \in \mcS$ and
$t^X:X \otimes (- \otimes -) \Ra (X \otimes -) \otimes -$,
$b^X:X \otimes (- \otimes -) \Ra - \otimes (X \otimes -)$
are natural transformations between endofunctors of $\mcS\times\mcS$, such that the following diagrams commute for all $U,V,W\in\mcS$:
\begin{equation}
\label{eq:AS_defn_diag1}
\begin{tikzpicture}[baseline={([yshift=-.5ex]current bounding box.center)}]
\node (P1) at (-5,0) {$X(U(VW))$};
\node (P2) at (-3,2) {$(XU)(VW)$};
\node (P3) at (-1,0) {$((XU)V)W$};
\node (P4) at (-4.8,-2) {$X((UV)W)$};
\node (P5) at (-1.2,-2) {$(X(UV))W$};

\path[commutative diagrams/.cd,every arrow,every label]
(P1) edge node {$t^X_{U,VW}$} (P2)
(P2) edge node {$a^{-1}_{XU,V,W}$} (P3)
(P1) edge node[swap] {$\id_X \otimes a^{-1}_{U,V,W}$} (P4)
(P4) edge node[swap] {$t^X_{UV,W}$} (P5)
(P5) edge node[swap] {$t^X_{U,V} \otimes \id_W$} (P3);
\end{tikzpicture},
\end{equation}
\begin{equation}
\label{eq:AS_defn_diags23}
\begin{tikzpicture}[baseline={([yshift=-.5ex]current bounding box.center)}]
\node (P1) at (-5,0) {$X(U(VW))$};
\node (P2) at (-4.8,2) {$U(X(VW))$};
\node (P3) at (-1.2,2) {$U((XV)W)$};
\node (P4) at (-1,0) {$(U(XV))W$};
\node (P5) at (-4.8,-2) {$X((UV)W)$};
\node (P6) at (-1.2,-2) {$(X(UV))W$};

\path[commutative diagrams/.cd,every arrow,every label]
(P1) edge node {$b^X_{U,VW}$} (P2)
(P2) edge node[above=4pt] {$\id_U \otimes t^X_{V,W}$} (P3)
(P3) edge node {$a^{-1}_{U,XV,W}$} (P4)
(P1) edge node[swap] {$\id_X \otimes a^{-1}_{U,V,W}$} (P5)
(P5) edge node[swap] {$t^X_{UV,W}$} (P6)
(P6) edge node[swap] {$b^X_{U,V}\otimes \id_W$} (P4);
\end{tikzpicture}
\quad  , \quad
\begin{tikzpicture}[baseline={([yshift=-.5ex]current bounding box.center)}]
\node (P1) at (-5,0) {$X(U(VW))$};
\node (P2) at (-4.8,2) {$U(X(VW))$};
\node (P3) at (-1.2,2) {$U(V(XW))$};
\node (P4) at (-1,0) {$(UV)(XW)$};
\node (P5) at (-3,-2) {$X((UV)W)$};

\path[commutative diagrams/.cd,every arrow,every label]
(P1) edge node {$b^X_{U,VW}$} (P2)
(P2) edge node[above=4pt] {$\id_U\otimes b^X_{V,W}$} (P3)
(P3) edge node {$a^{-1}_{U,V,XW}$} (P4)
(P1) edge node[swap] {$\id_X \otimes a^{-1}_{U,V,W}$} (P5)
(P5) edge node[swap] {$b^X_{UV,W}$}(P4);
\end{tikzpicture};
\end{equation}
\item
\it{morphisms:}
$\varphi:(X,t^X,b^X) \ra (Y,t^Y,b^Y)$ is a natural transformation\\ $\varphi:X \otimes - \Ra Y \otimes -$, such that the following diagrams commute for all $U,V\in\mcS$
\begin{equation}
\label{eq:AS_defn_morphs}
\begin{tikzpicture}[baseline={([yshift=-.5ex]current bounding box.center)}]
\node (P1) at (-3,0) {$X(UV)$};
\node (P2) at (-1.5,1) {$Y(UV)$};
\node (P3) at (0,0) {$(YU)V$};
\node (P4) at (-1.5,-1) {$(XU)V$};

\path[commutative diagrams/.cd,every arrow,every label]
(P1) edge node {$\varphi_{UV}$} (P2)
(P2) edge node {$t^Y_{U,V}$} (P3)
(P1) edge node[swap] {$t^X_{U,V}$} (P4)
(P4) edge node[swap] {$\varphi_{U} \otimes \id_V$} (P3);
\end{tikzpicture}, \quad
\begin{tikzpicture}[baseline={([yshift=-.5ex]current bounding box.center)}]
\node (P1) at (-3,0) {$X(UV)$};
\node (P2) at (-1.5,1) {$Y(UV)$};
\node (P3) at (0,0) {$U(YV)$};
\node (P4) at (-1.5,-1) {$U(XV)$};

\path[commutative diagrams/.cd,every arrow,every label]
(P1) edge node {$\varphi_{UV}$} (P2)
(P2) edge node {$b^Y_{U,V}$} (P3)
(P1) edge node[swap] {$b^X_{U,V}$} (P4)
(P4) edge node[swap] {$\id_U \otimes \varphi_V$} (P3);
\end{tikzpicture}.
\end{equation}
\end{itemize}
\end{defn}

\begin{prp}
\label{prp:ZS_equiv_AS}
The functor $E:\mcZ(\mcS) \ra \mcA(\mcS)$, acting
\begin{itemize}
\item
\it{on objects:}
$E(X,\g) := (X, t^X , b^X)$, where for all $U,V\in\mcS$
\begin{align}
& t^X_{U,V} := \left[ X(UV) \xra{a^{-1}_{X,U,V}} (XU)V \right] \ ,
\nonumber\\
& b^X_{U,V} := \left[ X(UV) \xra{a^{-1}_{X,U,V}} (XU)V \xra{\g_U \otimes \id_V} (UX)V \xra{a_{U,X,V}} U(XV) \right] \ ;
\label{eq:Functor-E-on-obj}
\end{align}
\item
\it{on morphisms:}
$E\big( ~ [(X,\g) \xra{f} (Y,\d)] ~ \big) :=
\{X\otimes U \xra{f \otimes \id_U} Y \otimes U\}_{U\in\mcS}$.
\end{itemize}
is a linear equivalence.
\end{prp}
\begin{proof}
It is easy to see that $E(X,\gamma)$ is indeed an object in $\mcA(\mcS)$ and that $E(f)$ is a morphism in $\mcA(\mcS)$.
In the remainder of the proof we show that $E$ is essentially surjective and fully faithful.

As a preparation, given an object $(X,t,b)\in \mcA(\mcS) \in \mcA(\mcS)$ we derive some properties of $t$ and $b$. For $V,W\in\mcS$, consider the following diagram, whose ingredients we proceed to explain:
\begin{equation}
\label{eq:E-proof_pentagon_diag}
\begin{tikzpicture}[baseline={([yshift=-.5ex]current bounding box.center)}]
\node (P1) at (-6,0) {$X(\opid(VW))$};
\node (P2) at (-3,1) {$(X\opid)(VW)$};
\node (P3) at (0,0) {$((X\opid)V)W$};
\node (P4) at (-4.5,-1) {$X((\opid V)W)$};
\node (P5) at (-1.5,-1) {$(X(\opid V))W$};

\node (A1) at (-8.5,0) {$X(VW)$};
\node (A2) at (-3,2.5) {$X(VW)$};
\node (A3) at (2.5,0) {$(XV)W$};
\node (A4) at (-4.5,-2.5) {$X(VW)$};
\node (A5) at (-1.5,-2.5) {$(XV)W$};

\path[commutative diagrams/.cd,every arrow,every label]
(P1) edge node {$t_{\opid,VW}$} (P2)
(P2) edge node {$a^{-1}_{X\opid,V,W}$} (P3)
(P1) edge node[swap,pos=0.0] {$\id_X \otimes a^{-1}_{\opid,V,W}$} (P4)
(P4) edge node[swap] {$t_{\opid V,W}$} (P5)
(P5) edge node[swap,pos=1.0] {$t_{\opid, V} \otimes \id_W$} (P3);

\path[commutative diagrams/.cd,every arrow,every label]
(A1) edge node {$\widehat{t_{\opid,VW}}$} (A2)
(A2) edge node {$a^{-1}_{X,V,W}$} (A3)
(A1) edge node[swap] {$\id$} (A4)
(A4) edge node[swap] {$t_{V,W}$} (A5)
(A5) edge node[swap] {$\widehat{t_{\opid, V}}\otimes \id_W$} (A3);

\path[commutative diagrams/.cd,every label]
(A1) edge node {$\sim$} (P1)
(A2) edge node[rotate=90,above=-10pt] {$\sim$} (P2)
(A3) edge node[swap] {$\sim$} (P3)
(A4) edge node[rotate=90,above=-10pt] {$\sim$}(P4)
(A5) edge node[rotate=90,above=-10pt] {$\sim$} (P5);
\end{tikzpicture}.
\end{equation}
We abbreviate $t^X$ by $t$, and we use the following notation for all $U\in\mcS$:
\begin{align}\nonumber
& \widehat{t_{\opid, U}} := \left[ XU \xra{\id_X\otimes l^{-1}_U} X(\opid U) \xra{t_{\opid, U}} (X\opid)U \xra{r_X \otimes \id_U} XU \right],\\
& \widehat{t_{U,\opid}}  := \left[ XU \xra{\id_X \otimes r_U^{-1}} X(U\opid) \xra{t_{U,\opid}} (XU)\opid \xra{r_{XU}} XU \right],\\ \nonumber
& \widehat{b_{\opid,U}}  := \left[ XU \xra{\id_X\otimes l^{-1}_U} X(\opid U) \xra{b_{\opid,U}} \opid (XU) \xra{l_{XU}} XU \right],\\ \nonumber
& \widehat{b_{U,\opid}}  := \left[ XU \xra{\id_X \otimes r_U^{-1}} X(U\opid) \xra{b_{U,\opid}} U(X\opid) \xra{\id_U\otimes r_X} UX \right].
\end{align}
This notation will be used in the remainder of this section, too.

By taking $U=\opid$ in \eqref{eq:AS_defn_diag1}, the inner pentagon	in \eqref{eq:E-proof_pentagon_diag} 
commutes and all squares commute by definition of $\widehat{t_{\opid,U}}$, by naturality or by monoidal coherence, and hence the outer pentagon commutes as well.
Leaving out the identity edge, we get the following commutative diagram for all $V,W\in\mcS$:
\begin{equation}
\label{eq:E-proof_diag1}
\begin{tikzpicture}[baseline={([yshift=-.5ex]current bounding box.center)}]
\node (P1) at (-3,0) {$X(VW)$};
\node (P2) at (-1.5,1) {$X(VW)$};
\node (P3) at (0,0) {$(XV)W$};
\node (P4) at (-1.5,-1) {$(XV)W$};

\path[commutative diagrams/.cd,every arrow,every label]
(P1) edge node {$\widehat{t_{\opid,VW}}$} (P2)
(P2) edge node {$a^{-1}_{X,V,W}$} (P3)
(P1) edge node[swap] {$t_{V,W}$} (P4)
(P4) edge node[swap] {$\widehat{t_{\opid, V}}\otimes \id_W$} (P3);
\end{tikzpicture}.
\end{equation}
Similarly, by taking $V=\opid$ and $W=\opid$ in \eqref{eq:AS_defn_diag1} one in the end gets the following two commuting diagrams:
\begin{equation}
\label{eq:E-proof_aux_diags}
\forall U,W\in\mcS:
\begin{tikzpicture}[baseline={([yshift=-.5ex]current bounding box.center)}]
\node (P1) at (0,0)  {$X(UW)$};
\node (P2) at (2.5,0)  {$(XU)W$};
\node (P3) at (0,-2) {$(XU)W$};

\path[commutative diagrams/.cd,every arrow,every label]
(P1) edge node {$t_{U,W}$} (P2)
(P1) edge node[swap] {$t_{U,W}$} (P3)
(P2) edge node {$\widehat{t_{U,\opid}}\otimes\id_W$} (P3);
\end{tikzpicture},
\forall U,V\in\mcS:
\begin{tikzpicture}[baseline={([yshift=-.5ex]current bounding box.center)}]
\node (P1) at (0,0)  {$X(UV)$};
\node (P2) at (2.5,0)  {$(XU)V$};
\node (P3) at (2.5,-2) {$X(UV)$};

\path[commutative diagrams/.cd,every arrow,every label]
(P1) edge node {$t_{U,V}$} (P2)
(P1) edge node[swap] {$\widehat{t_{UV,\opid}}$} (P3)
(P3) edge node {$t_{U,V}$} (P2);
\end{tikzpicture}.
\end{equation}
These diagrams imply that for all $U\in\mcS$ one has $\widehat{t_{U,1}} = \id_U$.

Repeating the above procedure of setting individual objects to $\opid$ also for two diagrams in \eqref{eq:AS_defn_diags23} yields three more conditions.
Namely, for all $U,V,W\in\mcS$ one has $\widehat{b_{\opid, U}} = \id_U$ and the following diagrams commute:
\begin{equation}
\label{eq:E-proof_diags23}
\begin{tikzpicture}[baseline={([yshift=-.5ex]current bounding box.center)}]
\node (P1) at (-6,0) {$X(UW)$};
\node (P2) at (-4.5,1) {$U(XW)$};
\node (P3) at (-1.5,1) {$U(XW)$};
\node (P4) at (0,0) {$(UX)W$};
\node (P5) at (-3,-1) {$(XU)W$};

\path[commutative diagrams/.cd,every arrow,every label]
(P1) edge node[pos=0.2] {$b_{U,W}$} (P2)
(P2) edge node[above=3pt] {$\id_U \otimes \widehat{t_{\opid, W}}$} (P3)
(P3) edge node[pos=0.7] {$a^{-1}_{U,X,W}$} (P4)
(P1) edge node[swap] {$t_{U,W}$} (P5)
(P5) edge node[swap] {$\widehat{b_{U, \opid}}\otimes\id_W$}(P4);
\end{tikzpicture},
\begin{tikzpicture}[baseline={([yshift=-.5ex]current bounding box.center)}]
\node (P1) at (-2.5,1.5) {$X(UV)$};
\node (P2) at (0,1.5) {$U(XV)$};
\node (P3) at (0,0) {$U(VX)$};
\node (P4) at (-2.5,0) {$(UV)X$};

\path[commutative diagrams/.cd,every arrow,every label]
(P1) edge node {$b_{U,V}$} (P2)
(P2) edge node {$\id_U\otimes \widehat{b_{V,\opid}}$} (P3)
(P3) edge node {$a^{-1}_{U,V,X}$} (P4)
(P1) edge node {$\widehat{b_{UV,\opid}}$} (P4);
\end{tikzpicture}.
\end{equation}

Let $\eta_U := [XU \xra{\widehat{t_{\opid, U}}} XU]$ and $\g_U := [XU \xra{\eta_U^{-1}} XU \xra{\widehat{b_{U,\opid}}} UX]$ for all $U\in\mcS.$
We claim that $\g$ is a half-braiding for $X$ and that $\eta$ is an isomorphism $(X,t,b) \ra E(X,\g)$ in $A(\mcS)$.

We start by showing that $\eta$ is indeed a morphism in $A(\mcS)$.
First note that \eqref{eq:E-proof_diag1} now reads
\begin{equation}
\label{eq:t_UV-and-a^-1}
t_{U,V} = \left[X(UV) \xra{\eta_{UV}} X(UV) \xra{a^{-1}_{X,U,V}} (XU)V \xra{\eta^{-1}_U \otimes \id_V} (XU)V\right] \ .
\end{equation}
Together with the definition of $E(X,\gamma)$ in \eqref{eq:Functor-E-on-obj} we see that this is precisely the first condition in \eqref{eq:AS_defn_morphs}. 
Plugging \eqref{eq:t_UV-and-a^-1} into the first diagram in \eqref{eq:E-proof_diags23} results in
\begin{equation}
\label{eq:b_UV-and-gamma}
b_{U,V} = \left[
\begin{array}{rl}
X(UV) &\xra{\eta_{UV}} X(UV) \xra{a^{-1}_{X,U,V}} (XU)V \xra{\g_U\otimes\id_V} (UX)V\\
&\xra{a_{U,X,V}} U(XV) \xra{\id_U\otimes\eta^{-1}_V} U(XV)
\end{array}
\right] \ .
\end{equation}
This is precisely the second condition in \eqref{eq:AS_defn_morphs}.

Checking that $\g$ satisfies the hexagon condition is now a direct consequence of plugging \eqref{eq:b_UV-and-gamma} into the second diagram in \eqref{eq:E-proof_diags23}.

Altogether, this shows that $E$ is essentially surjective. 

To get that $E$ is fully faithful, let $\varphi : E(X,\g)\ra E(Y,\d)$ be a morphism in $A(\mcS)$. Setting $U = \opid$ in the first condition in \eqref{eq:AS_defn_morphs} yields that for all $V\in\mcS$ one has $\varphi_V = \widehat{\varphi_\opid} \otimes \id_V$, where
\begin{equation}
\widehat{\varphi_\opid} := \left[ X \xra{r_X^{-1}} X\opid \xra{\varphi_\opid} Y\opid \xra{r_Y} Y \right] \ .
\end{equation}
Setting $V = \opid$ in the second condition in \eqref{eq:AS_defn_morphs} 
shows that $\widehat{\varphi_\opid}$ commutes with the half-braidings $\g$ and $\d$.
Altogether, $\varphi$ is in the image of $E$.
\end{proof}
\subsubsection{The functor $\boldsymbol{D}$ from $\boldsymbol{\mcA(\mcS)}$ to $\boldsymbol{\mcC_{\opA^\mcS}}$}
\label{subsubsec:functorD}

In this subsection we define a functor $D : \mcA(\mcS) \to \mcC_{\opA^\mcS}$. We start by defining $D$ on objects. Let $(X,t,b) \in \mcA(\mcS)$ and denote the components of $D(X,t,b)$ by
\begin{equation}\label{eq:Ddef-value-on-obj}
D(X,t,b) =: (M, \tau_1, \tau_2, \overline{\tau_1}, \overline{\tau_2}) \ .
\end{equation}
We will go through the definition of the constituents step by step, starting with the $A$-$A$-bimodule $M$.

\medskip

For $n\ge 1$, an $A^{\otimes n}$-module is an $\mcI^{\times n}$-graded vector space and a morphism between modules is a grade-preserving linear map.
In particular, an $A$-$A$-bimodule $M$ is a vector space with a decomposition $M = \bigoplus_{i,j\in\mcI} M_{ij}$, where for $M_{ij}$ only the $\mcS(i,i)$-$\mcS(j,j)$ action is non-trivial.
For the bimodule $M$ in \eqref{eq:Ddef-value-on-obj} we set
\begin{equation}
\label{eq:Mlk_is_SlXk}
M = \bigoplus_{i,j\in\mcI} M_{ij}
\quad \text{with} \quad
M_{ij} = \mcS(i, Xj) \ ,
\end{equation}
with action of $\mcS(i,i)$ (from the left) and $\mcS(j,j)$ (from the right) given by pre- and post-composition, respectively.

Next we turn to defining $\tau_i$ and $\overline{\tau_i}$. We will need two ingredients. The first are certain $A$-$AA$-bimodule isomorphisms $\sigma_x$, $x=0,1,2$, which are defined as
\begin{align}
M \otimes_0 T
=
\bigoplus_{l,i,j,a\in\mcI} \mcS(l, Xa) \otimes_{\opk} \mcS(a, ij)
&~\xrightarrow{~\sigma_0~}~
\bigoplus_{l,i,j\in\mcI} \mcS(l, X(ij))
\nonumber \\
[l \xra{f} Xa] \otimes_\opk [a \xra{g} ij] 
&~\longmapsto~
[l \xra{(\id_X \otimes g) \circ f} X(ij)] \ ,
\nonumber \\[1em]
T \otimes_1 M
=
\bigoplus_{l,i,j,a\in\mcI} \mcS(l, aj) \otimes_{\opk} \mcS(a, Xi)
&~\xrightarrow{~\sigma_1~}~
\bigoplus_{l,i,j\in\mcI} \mcS(l, (Xi)j)
\nonumber \\
[l \xra{f} aj] \otimes_\opk [a \xra{g} Xi] 
&~\longmapsto~
[l \xra{(g \otimes \id_j) \circ f} (Xi)j)] \ ,
\nonumber \\[1em]
T \otimes_2 M
=
\bigoplus_{l,i,j,a\in\mcI} \mcS(l, ia) \otimes_{\opk} \mcS(a, Xj)
&~\xrightarrow{~\sigma_2~}~
\bigoplus_{l,i,j\in\mcI} \mcS(l, i(Xj))
\nonumber \\
[l \xra{f} ia] \otimes_\opk [a \xra{g} Xj] 
&~\longmapsto~
[l \xra{(\id_i \otimes g) \circ f} i(Xj))] \ .
\label{eq:sigma_i-def}
\end{align}

To describe the second ingredient, it will be useful to relate linear maps between morphism spaces in $\mcS$ to actual morphisms in $\mcS$ as described in the following remark.

\begin{rem}
	\label{rem:homs_to_morphs}
Let $\euT^n_B:\mcS \times \cdots \times \mcS \ra \mcS$ be the functor which takes the $n$-fold tensor product with a given bracketing $B$.
Consider the $\mcI^{\times (n+1)}$-graded vector space $V_B := \bigoplus_{l,i_1,\dots,i_n} \mcS(l, \euT^n_B(i_1, \dots, i_n))$.
For two bracketings $B$, $B'$, one has a linear isomorphism
\begin{equation*}
\big\{ \text{natural transformations } \euT^n_B \Ra \euT^n_{B'} \big\}
	~\xrightarrow{~\sim~}~
\big\{ \text{graded linear maps } V_B \ra V_{B'} \big\} \ ,
\end{equation*}
given by post-composition. That is, it takes a natural transformation $\varphi$ to the graded linear map
\begin{equation}
	\left[ l \xra{f} \euT^n_B(i_1, \dots, i_n) \right]
	~\longmapsto~
	\left[ l \xra{f} \euT^n_B(i_1, \dots, i_n) \xra{\varphi_{i_1,\dots,i_n}} \euT^n_{B'}(i_1, \dots, i_n) \right] \ .
\end{equation}
This is easily generalised for functors obtained from $\euT^n_B$ by fixing some of the arguments.
\end{rem}

Recall the natural transformations $t$, $b$ that form part of the object $(X,t,b)$ on which we are defining the functor $D$.
The second ingredient needed to define $\tau_i$, $\overline{\tau_i}$ are four families of morphisms $(\tau_i')_{UV}$, $(\overline{\tau_i}')_{UV}$ in $\mcS$ which are natural in $U,V \in \mcS$:
\begin{align}
\big[ X(UV) \xra{(\tau_1')_{UV}} (XU)V \big] 
&~:=~ 
\big[ X(UV) \xra{t_{UV}} (XU)V \xra{\psi^{-2}_{XU}\otimes\id_V} (XU)V \big]
\ , 
\nonumber\\ 
\big[ X(UV) \xra{(\tau_2')_{UV}} U(XV) \big] 
&~:=~ 
\big[ X(UV) \xra{b_{UV}} U(XV) \xra{\id_U\otimes \psi^{-2}_{XV}} U(XV) \big]
\ ,
\nonumber\\ 
\big[ (XU)V \xra{(\overline{\tau_1}')_{UV}} X(UV) \big]
&~:=~ 
\big[ (XU)V \xra{t^{-1}_{UV}} X(UV) \xra{\id_X \otimes \psi^{-2}_{UV}} X(UV) \big] 
\ ,
\nonumber\\ 
\big[U(XV) \xra{(\overline{\tau_2}')_{UV}} X(UV)\big]
&~:=~ 
\big[U(XV) \xra{b^{-1}_{UV}} X(UV) \xra{\id_X \otimes \psi^{-2}_{UV}} X(UV)\big] 
\ .
\label{eq:tau'-from-t-b}
\end{align}
Combining these two ingredients, we define $\tau_i$, $\overline{\tau_i}$ in \eqref{eq:Ddef-value-on-obj} to be:
\begin{align}
\tau_1 &:= \big[\, M \otimes_0 T 
\xrightarrow{\sigma_0} \mcS(l, X(ij))
\xrightarrow{(\tau_1')_{ij}\circ (-)} \mcS(l, (Xi)j)
\xrightarrow{\sigma_1^{-1}} T \otimes_1 M \big] \ ,
\nonumber \\
\tau_2 &:= \big[\, M \otimes_0 T 
\xrightarrow{\sigma_0} \mcS(l, X(ij))
\xrightarrow{(\tau_2')_{ij}\circ (-)} \mcS(l, i(Xj))
\xrightarrow{\sigma_2^{-1}} T \otimes_2 M \big] \ ,
\nonumber \\
\overline{\tau_1} &:= \big[\, T \otimes_1 M 
\xrightarrow{\sigma_1} \mcS(l, (Xi)j)
\xrightarrow{(\overline{\tau_1}')_{ij}\circ (-)} \mcS(l, X(ij))
\xrightarrow{\sigma_0^{-1}} M \otimes_0 T \big] \ ,
\nonumber \\
\overline{\tau_2} &:= \big[\, T \otimes_2 M 
\xrightarrow{\sigma_2} \mcS(l, i(Xj))
\xrightarrow{(\overline{\tau_2}')_{ij}\circ (-)} \mcS(l, X(ij))
\xrightarrow{\sigma_0^{-1}} M \otimes_0 T \big] \ .
\label{eq:tau-from-tau'}
\end{align}
The verification that these morphisms satisfy \eqrefT{1}-\eqrefT{7} will be part of the proof of Proposition~\ref{prp:AS_equiv_CAS} below.

The action of $D$ on a morphism $\varphi:(X,t^X,b^X) \ra (Y,t^Y,b^Y)$ in $\mcA(\mcS)$ is
\begin{equation}\label{eq:Ddef-morph}
D(\varphi) := \Big[
D(X,t^X,b^X) = \bigoplus_{i,j\in\mcI} \mcS(i, Xj) 
\xrightarrow{~\varphi_j \circ (-)~} \bigoplus_{i,j\in\mcI} \mcS(i, Yj) = D(Y,t^Y,b^Y) \Big] \ .
\end{equation} 

\begin{prp}
\label{prp:AS_equiv_CAS}
The functor $D:\mcA(\mcS) \ra \mcC_{\opA^\mcS}$ is well-defined and a linear equivalence.
\end{prp}

\begin{proof}
The proof that $D(X,t,b)$ is indeed an object in $\mcC_{\opA^\mcS}$ is a little tedious and will be given in Subsection~\ref{subsubsec:taus-satisfy-Tcond} below. For now we assume that this has been done and continue with the remaining points.

To see that $D(\varphi) : D(X,t^X,b^X) \ra D(Y,t^Y,b^Y)$ is a morphism in $\mcC_{\opA^\mcS}$ we have to verify the identities in \eqref{eq:M}. We will demonstrate this for $\tau_1$ as an example. Denote the underlying $A$-$A$-bimodules of $D(X,t^X,b^X)$ and $D(Y,t^Y,b^Y)$ as $M$ and $N$, respectively, and consider the following diagram:
\begin{equation}\label{eq:condition-M-vs-cond-for-t}
\begin{tikzpicture}[baseline={([yshift=-.5ex]current bounding box.center)}]
\node (O1) at (0,4.5) {$M \otimes_0 T$};
\node (O2) at (12,4.5) {$T \otimes_1 M$};
\node (O3) at (0,0) {$N \otimes_0 T$};
\node (O4) at (12,0) {$T \otimes_1 N$};
\node (I1) at (2,3) {\small $\oplus \mcS(l,X(ij))$};
\node (I2) at (6,3) {\small $\oplus \mcS(l,(Xi)j)$};
\node (I3) at (10,3) {\small $\oplus \mcS(l,(Xi)j)$};
\node (I4) at (2,1.5) {\small $\oplus \mcS(l,Y(ij))$};
\node (I5) at (6,1.5) {\small $\oplus \mcS(l,(Yi)j)$};
\node (I6) at (10,1.5) {\small $\oplus \mcS(l,(Yi)j)$};

\path[commutative diagrams/.cd,every arrow,every label]
(I1) edge node {\small $\oplus (t^X_{ij})_*$} (I2)
(I2) edge node {\small $\oplus (\psi_{Xi}^{-2} \otimes \id)_*$} (I3)
(I4) edge node[swap] {\small $\oplus (t^Y_{ij})_*$} (I5)
(I5) edge node[swap] {\small $\oplus (\psi_{Yi}^{-2} \otimes \id)_*$} (I6)
(I1) edge node[swap] {\small $\oplus (\varphi_{ij})_*$} (I4)
(I2) edge node[swap] {\small $\oplus (\varphi_{i}\otimes \id)_*$} (I5)
(I3) edge node[swap] {\small $\oplus (\varphi_{i}\otimes \id)_*$} (I6);

\path[commutative diagrams/.cd,every arrow,every label]
(O1) edge node[swap] {\small $\sigma_0$} (I1)
(O3) edge node {\small $\sigma_0$} (I4)
(O2) edge node {\small $\sigma_1$} (I3)
(O4) edge node[swap] {\small $\sigma_1$} (I6);

\path[commutative diagrams/.cd,every arrow,every label]
(O1) edge node[swap] {\small $\tau_1^M$} (O2)
(O3) edge node {\small $\tau_1^N$} (O4)
(O1) edge node[rotate=90,above] {\small $D(\varphi)\otimes_0\id$} (O3)
(O2) edge node[rotate=90,above] {\small $\id \otimes_1 D(\varphi)$} (O4);
\end{tikzpicture}
\end{equation}
Here, all direct sums run over $i,j,l \in \mcI$. The notation $(-)_*$ stands for post-composition with the corresponding morphism. 
The left innermost square commutes by \eqref{eq:AS_defn_morphs}, and the right innermost square commutes by naturality of $\psi$. The top and bottom squares are just the definition of $\tau_1$ in \eqref{eq:tau'-from-t-b} and \eqref{eq:tau-from-tau'}. 
That the rightmost square commutes is immediate from the definition of $\sigma_1$ in \eqref{eq:sigma_i-def}, while for the leftmost square one needs to invoke in addition the naturality of $\varphi$.

\medskip

So far we have shown that the functor $D$ is well-defined. We now check that it is essentially surjective and fully faithful.

\medskip

Let $(M, \tau_1, \tau_2, \overline{\tau_1}, \overline{\tau_2})$ be an arbitrary object in $\mcC_{\opA^\mcS}$.
As above, we decompose $M = \bigoplus_{i,j\in\mcI} M_{ij}$, where for $M_{ij}$ only the $\mcS(i,i)$-$\mcS(j,j)$ action is non-trivial.
Since $\tau_1$ is an $A$-$AA$-bimodule isomorphism $M \otimes_0 T \xra{\sim} T \otimes_1 M$, we have a graded linear isomorphism
\begin{equation}
\label{eq:TV_tau1_form}
\tau_1: \bigoplus_{i,j,l,a\in\mcI} M_{la} \otimes_\opk \mcS(a,ij) 
\,\xra{~\sim~}\, 
\bigoplus_{i,j,l,b\in\mcI} \mcS(l,bj) \otimes_\opk M_{bi} \ ,
\end{equation}
Specialising to $i=\opid$ gives linear isomorphisms, for all $l,j\in \mcI$,
\begin{equation}
\label{eq:tau1_impl1}
 \underbrace{\bigoplus_{a\in\mcI} M_{la} \otimes_\opk \mcS(a, \opid j)}_{\cong M_{lj}} 
\,\xra{~\sim~}\, 
\bigoplus_{b\in\mcI} \mcS(l,bj) \otimes_\opk M_{b\opid} \ , \\
\end{equation}
Setting $X = \bigoplus_{b\in\mcI} b \otimes M_{b\opid} \in \mcS$, we see that this implies $M \cong \bigoplus_{l,j} \mcS(l, Xj)$ as $A$-$A$-bimodules. We may thus assume without loss of generality that in fact $M = \bigoplus_{l,j} \mcS(l, Xj)$ for some $X \in \mcS$.

Define $t,b$ by inverting the first two defining relations in each of \eqref{eq:tau'-from-t-b} and \eqref{eq:tau-from-tau'} (this is possible by Remark~\ref{rem:homs_to_morphs}). We need to verify that $t,b$ satisfy the conditions in \eqref{eq:AS_defn_diag1} and \eqref{eq:AS_defn_diags23}.

Consider condition \eqrefT{1} satisfied by $\tau_1$. Along the same lines as was done in \eqref{eq:condition-M-vs-cond-for-t}, one can translate \eqrefT{1} into an equality of two graded linear maps $\bigoplus \mcS(l,X(i(jk)))$ $\ra$ $\bigoplus \mcS(l, ((Xi)j)k)$. Both of these maps are given by post-composition, resulting in a commuting diagram of morphisms in $\mcS$, for all $i,j,k$:
\begin{equation}\label{eq:D-functor-T1-cond}
\begin{tikzpicture}[baseline={([yshift=-.5ex]current bounding box.center)}]
\node (P1) at (-5,1) {$X(i(jk))$};
\node (P2) at (0,4) {$(Xi)(jk)$};
\node (P3) at (5,1) {$((Xi)j)k$};
\node (P4) at (-4,-3) {$X((ij)k)$};
\node (P5) at (4,-3) {$(X(ij))k$};

\node (P12) at (-2.5,2.5) {$(Xi)(jk)$};
\node (P23) at (2.5,2.5) {$((Xi)j)k$};
\node (P14) at (-4.5,-1) {$X((ij)k)$};
\node (P53) at (4.5,-1) {$((Xi)j)k$};

\node (P451) at (-1.3,-3) {$X((ij)k)$};
\node (P452) at (1.3,-3) {$(X(ij))k$};

\path[commutative diagrams/.cd,every arrow,every label] 
(P1)  edge node {$t_{i,jk}$} (P12)
(P12) edge node {$\psi^{-2}_{Xi} \otimes \id_{jk}$} (P2) 
(P2)  edge node {$a^{-1}_{Xi,j,k}$} (P23) 
(P23) edge node {$\psi^{-2}_{(Xi)j} \otimes \id_k$} (P3);

\path[commutative diagrams/.cd,every arrow,every label]
(P1)   edge node[swap] {$\id_X \otimes a^{-1}_{ijk}$} (P14)
(P14)  edge node[swap] {$\id_X \otimes (\psi^{-2}_{ij} \otimes \id_k)$} (P4)
(P4)   edge node[rotate=60,pos=0.7] {\makebox[0pt][l]{$\id_X \otimes (\psi^2_{ij} \otimes \id_k)$}} (P451)
(P451) edge node[rotate=60,pos=0.7] {\makebox[0pt][l]{$t_{(ij),k}$}} (P452)
(P452) edge node[rotate=60,pos=0.7] {\makebox[0pt][l]{$\psi^{-2}_{X(ij)}\otimes \id_k$}} (P5)
(P5)   edge node[swap] {$t_{ij} \otimes \id_k$} (P53)
(P53)  edge node[swap] {$(\psi^{-2}_{Xi} \otimes \id_j)\otimes \id_k$} (P3);
\end{tikzpicture}
\end{equation}
Since $t$, $a$ and $\psi$ are natural transformations, one can cancel all arrows with $\psi$, which then yields precisely the diagram \eqref{eq:AS_defn_diag1}.
Similarly, \eqrefT{2}, \eqrefT{3} give the two diagrams in \eqref{eq:AS_defn_diags23}.

It remains to show that $D$ is fully faithful. Faithfulness is clear from \eqref{eq:Ddef-morph}. For fullness, let $f : D(X,t^X,b^X) \to  D(Y,t^Y,b^Y)$ be a morphism in $\mcC_{\opA^\mcS}$. By Remark~\ref{rem:homs_to_morphs}, $f$ is given by post-composition with a natural transformation $\varphi:X \otimes - \Ra Y \otimes -$. The identities \eqref{eq:M} impose that the two diagrams in \eqref{eq:AS_defn_morphs} commute. Thus $\varphi$ is a morphism in $\mcA(\mcS)$ and $f = D(\varphi)$. 
\end{proof}

\begin{cor}
The orbifold datum $\opA^\mcS$ is simple.
\end{cor}

\begin{proof}
By Proposition~\ref{prp:AS_equiv_CAS}, the functor $D:\mcA(\mcS) \ra \mcC_{\opA^\mcS}$ is a linear equivalence. 
Since $\mcC_{\opA^\mcS}$ is semisimple (Proposition~\ref{prp:CA-fin-ssi}), so is $\mcA(\mcS)$. Any object of the form $(\opid_{\mcS},t,b)$ is simple in $\mcA(\mcS)$, as $\opid_{\mcS}$ is simple in $\mcS$. For an appropriate choice of $t,b$ we have $D(\opid_{\mcS},t,b) \cong \opid_{\mcC_{\opA^\mcS}}$, the tensor unit of $\mcC_{\opA^\mcS}$. Using once more that $D$ is an equivalence, we conclude that $\opid_{\mcC_{\opA^\mcS}}$ is simple in $\mcC_{\opA^\mcS}$.
\end{proof}

\subsubsection{Conditions on $\boldsymbol{T}$-crossings}
\label{subsubsec:taus-satisfy-Tcond}

Here we complete the proof of Proposition~\ref{prp:AS_equiv_CAS} by showing that $D(X,t,b)$ from \eqref{eq:Ddef-value-on-obj} satisfies conditions \eqrefT{1}--\eqrefT{7}.

For condition \eqrefT{1}, the computation is the same as in the proof of essential surjectivity of $D$, just in the opposite direction, i.e.\ one starts by writing \eqref{eq:AS_defn_diag1} as \eqref{eq:D-functor-T1-cond}. Analogously, \eqref{eq:AS_defn_diags23} produces \eqrefT{2}, \eqrefT{3}.

Conditions \eqrefT{4} and \eqrefT{5} are straightforward to check from the definitions \eqref{eq:tau'-from-t-b} and \eqref{eq:tau-from-tau'}.

Since \eqrefT{6} and \eqrefT{7} involve duals, it is helpful to express the (vector space) dual bimodule $M^*$ of $M = \bigoplus_{l,a} \mcS(l,Xa)$ 
in terms of the bimodule $M^\vee := \bigoplus_{l,a} \mcS(l,X^*a)$. Given a basis $\{\mu\}$ of $\mcS(l,Xa)$, we get the basis $\{\mu^*\}$ of the dual vector space $\mcS(l,Xa)^*$ and the basis  $\{\bar\mu\}$ of $\mcS(Xa,l)$, which is dual to $\{\mu\}$ with respect to the composition pairing. Let us fix an isomorphism $M^* \to M^\vee$ as follows:
\begin{equation}
\label{eq:Mstar_Mvee}
\zeta : M^* \longrightarrow M^\vee
~ , ~~
\mu^* \longmapsto 
\frac{|a|}{|l|} \big[ a \xrightarrow{\sim} \opid a
\xrightarrow{ \coev_X \otimes \id } (XX^*)a \xrightarrow{\sim} X(X^*a)
 \xrightarrow{\bar\mu} Xl \big]
\end{equation}
Using $\zeta$, one can translate the evaluation and coevaluation maps from vector space duals to the new duals $M^\vee$. For example,
\begin{align}
\coev_M &:= \big[ A \xrightarrow{~~} M \otimes_A M^* 
\xrightarrow{ \id \otimes \zeta } M  \otimes_A M^\vee \big] 
~~,
\nonumber \\
\coevt_M &:= \big[ A \xrightarrow{~~} M^* \otimes_A M 
\xrightarrow{ \zeta \otimes \id } M^\vee \otimes_A M \big] \ ,
\end{align}
where the unlabelled arrow is the canonical coevaluation in vector spaces. Explicitly, this gives the $A$-$A$-bimodule maps
\begin{align}
	\ev_M &: M^\vee \otimes_A M \to A \ ,
	&\pic[1.5]{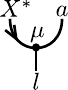} \otimes_\opk \pic[1.5]{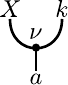} \mapsto \d_{l,k}\pic[1.5]{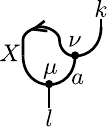},
	\nonumber\\
	\coev_M&: A \to M \otimes_A M^\vee \ ,
& \id_l \mapsto \sum_{a,\mu} \pic[1.5]{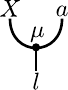} \otimes_\opk \left( \frac{|a|}{|l|} \pic[1.5]{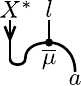} \right)
\nonumber\\
	\evt_M&: M \otimes_A M^\vee \to A \ ,
& \pic[1.5]{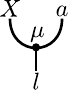} \otimes_\opk \pic[1.5]{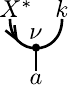} \mapsto \d_{l,k}
\frac{|l|}{|a|}
\pic[1.5]{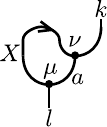}
\nonumber\\
	\coevt_M&: A \to M^\vee \otimes_A M  \ ,
&\id_l \mapsto \sum_{a,\mu} \left(\frac{|l|}{|a|} \pic[1.5]{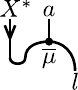}\right) \otimes_\opk \pic[1.5]{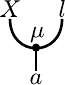}
\end{align}
The choice \eqref{eq:Mstar_Mvee} makes the expression for $\ev_M$ simpler but the other three duality maps still contain the dimension factors.
Using isomorphisms given by composition similar to those in \eqref{eq:sigma_i-def}, one can also write these maps as (by abuse of notation we keep the same names for the maps)
\begin{align}
	&\ev_M: \mcS(l,X^*(Xl)) \ra \mcS(l,l), &&
	\l \mapsto 
	\left[
	\begin{array}{l}
		l\xra{\l}X^*(Xl) \xra{a^{-1}_{X^*,X,l}}(X^*X)l \\
		\hfill \xra{\ev_X \otimes \id_l}\opid l \xra{l_l} l 
	\end{array}
	\right],\nonumber\\
	&\coev_M: \mcS(l,l) \ra \mcS(l,X(X^*l)), && 
	\mu \mapsto
	\left[
	\begin{array}{l}
		l      \xra{\mu}l\xra{l_l^{-1}}\opid l\xra{\coev_X}(XX^*)l\\
		\hfill	   \xra{a_{X,X^*,l}}X(X^*l)
	\end{array}
	\right],\nonumber\\
	&\evt_M: \mcS(l,X(X^*l)) \ra \mcS(l,l), &&
	\nu \mapsto
	\left[
	\begin{array}{l}
		l\xra{\nu}X(X^*l)\xra{\id_X\otimes\psi^{-2}_{X^* l}} X(X^* l)\\
		\hfill\xra{a^{-1}_{X,X^*,l}}(XX^*)l \xra{\evt_X \otimes \id_l}\opid l\\
		\hfill\xra{l_l} l \xra{\psi^2_l} l
	\end{array}
	\right],\nonumber\\
\label{eq:proper_ev-coevt}
	&\coevt_M: \mcS(l,l) \ra \mcS(l,X^*(Xl)), && 
	\xi \mapsto
	\left[
	\begin{array}{l}
		l\xra{\xi}l\xra{l_l^{-1}}\opid l\xra{\coevt_X}(X^*X)l\\
		\hfill\xra{a_{X^*,X,l}}X^*(Xl) \xra{\id_{X^*}\psi^{-2}_{Xl}} X^*(Xl)\\
		\hfill\xra{\id_{X^*} \otimes (\id_X \otimes \psi^2_l)} X^* (Xl)
	\end{array}
	\right].
\end{align}
For example to get the expression for $\coev_M$ one uses the identity:
\begin{equation}
\pic[1.5]{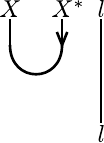} =
\sum_{a,\mu} \frac{|a|}{|l|} \pic[1.5]{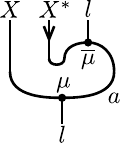}
\end{equation}

Note that these are dualities in $\mcACA$.
To obtain the dualities in $\mcC_{\opA^\mcS}$ some extra $\psi$-insertions are needed, see \eqref{eq:evs_evts_in_D}.

Given this reformulation of the duality morphisms, the verification of \eqrefT{6}, \eqrefT{7} works along the same lines as \eqrefT{1}--\eqrefT{3}.

\subsubsection{Ribbon structure of the composed functor}
\label{subsubsec:F-ribb}

Denoting the composed functor by $F := D \circ E$, we obtain the following corollary to Propositions~\ref{prp:ZS_equiv_AS} and \ref{prp:AS_equiv_CAS}.

\begin{cor}
\label{cor:ZS_equiv_D_linear}
The functor $F:\mcZ(\mcS) \ra \mcC_{\opA^\mcS}$, acting
\begin{itemize}[wide, labelwidth=!, labelindent=0pt]
\item
\it{on objects:}
$F(X,\g) := (\bigoplus_{k,l\in\mcI} \mcS(l,Xk), \tau_1, \tau_2, \overline{\tau_1}, \overline{\tau_2})$, where for all $i,j,l\in\mcI$ the $T$-crossings and their pseudo-inverses are (we omit writing out the isomorphisms $\sigma_i$ from \eqref{eq:sigma_i-def} explicitly)
\begin{align*}
&  \tau_1: \mcS(l,X(ij)) \ra \mcS(l, (Xi)j),
&&
\l \mapsto \left[l \xra{\l} X(ij) \xra{a^{-1}_{X,i,j}} (Xi)j \xra{\psi^{-2}_{Xi}\otimes\id_j} (Xi)j\right]\\
&  \tau_2: \mcS(l,X(ij)) \ra \mcS(l, i(Xj)),
&&
\mu\mapsto
\left[\begin{array}{rl}
	  l \xra{\mu} X(ij) & \xra{a^{-1}_{X,i,j}} (Xi)j \xra{\g_i\otimes\id_j} (iX)j \\
		  		        & \xra{a_{i,X,j}} i(Xj) \xra{\id_i\otimes\psi^{-2}_{Xj}} i(Xj)
\end{array}\right]\\
&  \overline{\tau_1}: \mcS(l, (Xi)j) \ra \mcS(l,X(ij)),
			&&
			\nu\mapsto \left[l \xra{\nu} (Xi)j \xra{a_{X,i,j}} X(ij) \xra{\id_X\otimes\psi^{-2}_{ij}} X(ij)\right]\\
			&  \overline{\tau_2}: \mcS(l,X(ij)) \ra \mcS(l, i(Xj)),
			&&
			\xi \mapsto
			\left[\begin{array}{rl}
				l \xra{\xi} i(Xj) & \xra{a^{-1}_{i,X,j}} (iX)j \xra{\g_i^{-1}\otimes\id_j} (Xi)j\\
				& \xra{a_{X,i,j}} X(ij) \xra{\id_X\otimes\psi^{-2}_{ij}} X(ij)
			\end{array}\right];
		\end{align*}
		\item
		\it{on morphisms:}
		$
		F([(X,\g) \xra{f} (Y,\d)]) :=
		\left[\begin{array}{rcl}
		\mcS(l,Xk) & \ra & \mcS(l,Yk) \quad \forall k,l\in\mcI\\
		{[} l \xra{g} Xk {]} &\mapsto & {[} l \xra{g}Xk \xra{f \otimes \id_k} Yk {]}
		\end{array} \right]
		$
	\end{itemize}
	is a linear equivalence.
\end{cor}
Recall that a monoidal structure consists of an isomorphism
\begin{equation}
F_0:\opid_{\mcC_{\opA^\mcS}} \xra{\sim} F(\opid_{\mcZ(\mcS)}),
\end{equation}
in $\mcC_{\opA^\mcS}$ as well as a collection of isomorphisms
\begin{equation}
F_2((X,\g), (Y,\d)):F(X,\g) \otimes_{\mcC_{\opA^\mcS}} F(Y,\d) \ra F(X \otimes_{\mcS} Y, \Gamma^{XY}),
\end{equation}
in $\mcC_{\opA^\mcS}$, natural in $(X,\g), (Y,\d) \in \mcZ(\mcS)$, satisfying the usual compatibility conditions (see e.g. \cite{TuVi}, Sec.1.4).
We set:
\begin{equation}
F_0: \bigoplus_{i\in\mcI} \mcS(i,i) \ra \bigoplus_{i\in\mcI} \mcS(i,\opid i), \quad
\left[i \xra{f} i\right] \mapsto \left[i \xra{f} i \xra{\psi^{-1}_i} i \xra{l_i^{-1}} \opid i\right].
\end{equation}
As in Section~\ref{subsubsec:functorD} we get the isomorphisms
\begin{equation}
F(X,\g) \otimes_{\mcC_{\opA^\mcS}} F(Y,\d) \cong \bigoplus_{l,r\in\mcI} \mcS(l,X(Yr)), \quad
F(X \otimes Y, \Gamma^{XY}) \cong \bigoplus_{l,r\in\mcI} \mcS(l, (XY)r).
\end{equation}
For all $l,r\in\mcI$, set
\begin{equation}
\label{eq:TV_monoidal_struct}
F_2((X,\g), (Y,\d)):
\left[l \xra{f} X(Yr)\right] \mapsto
\left[l \xra{f} X(Yr) \xra{\id_X \otimes \psi_{Yr}} X(Yr) 
\xra{a^{-1}_{X,Y,r}} (XY)r\right].
\end{equation}
One can check that they are indeed morphisms in $\mcC_{\opA^\mcS}$ and satisfy the compatibilities.
$F = (F,F_0,F_2)$ is therefore a monoidal equivalence.

\medskip

Recall, that $F$ is a braided functor if
\begin{equation}
\label{eq:braided_funct_cond}
F_2((Y,\d),(X,\g)) \circ c_{F(X,\g), F(Y,\d)} =
F(c_{(X,\g),(Y,\d)}) \circ F_2((X,\g), (Y,\d)).
\end{equation}
For $M = \bigoplus_{l,r\in\mcI} \mcS(l,Xr)$, $N = \bigoplus_{l,r\in\mcI} \mcS(l,Yr)$ with $T$-crossings 
$\tau_i^M$, $\tau_i^N$, 
$i=1,2$,
let us calculate the braiding morphism $c_{M,N} \in \mcC_{\opA^\mcS}$ explicitly.

Recall that the braiding in $\mcC_{\opA^\mcS}$ is obtained by taking the partial trace of the morphism $\Omega$  defined in \eqref{eq:almost_braiding}.
It amounts to a family of linear maps $\mcS(l, X(Y(ij))) \ra \mcS(l, Y(X(ij)))$, $i,j,l\in\mcI$, which are post-compositions with
\begin{align*}
B_{ij} &\,:=\, \frac{1}{\Dim\mcS} \cdot
\\&
\Big[\, 
X(Y(ij))
\xra{\id_X \otimes \psi_{Y(ij)}} X(Y(ij))
\xra{\id_X \otimes ({\tau'_1}^Y)_{ij}} X((Yi)j) 
\xra{({\tau'_2}^X)_{Yi,j}} (Yi)(Xj)
\\ &
\xra{\psi_{Yi}^2 \otimes \psi_{Xj}^2} (Yi)(Xj) 
\xra{(\id_Y\otimes\psi^2_i)\otimes(\id_X\otimes\psi^2_j)} (Yi)(Xj)
\xra{({\overline{\tau}_1'}^Y)_{i,Xj}} Y(i(Xj)) 
\\ &
\xra{\id_Y \otimes ({\overline{\tau}_2'}^X)_{ij}} Y(X(ij))
\xra{\id_Y \otimes \psi_{X(ij)}} Y(X(ij)) \, \Big] \ .
\end{align*}
We now need to trace the above morphism over $T$, for which we need the dual $T^*$. Similar to Section~\ref{subsubsec:taus-satisfy-Tcond} it is useful to work with $T^\vee := \bigoplus_{i,j,r\in\mcI} \mcS(ij,r)$ instead. 
Given a basis $\{\alpha\}$ of $\mcS(r,ij)$, the basis $\{\alpha^*\}$ of the dual vector space $\mcS(r,ij)^*$ and the composition-dual basis  $\{\bar\alpha\}$ of $\mcS(ij,r)$, we fix the isomorphism $T^* \to T^\vee$, $\alpha^* \mapsto \bar\alpha$. Using this isomorphism, the relevant evaluation and coevaluation maps are
\begin{align*}
	&[\evt_T: T \otimes_{1,2} T^\vee \ra A] &&=
	\left[
	\begin{array}{cccccc}
		\bigoplus_{i,j\in\mcI} & \mcS(l, ij) &\otimes_\opk &\mcS(ij, r) &\ra &\mcS(l,l)\\
		&f &\otimes_\opk  &g & \mapsto & \delta_{k,r} \, g \circ f
	\end{array}
	\right],\\
	&[\coev_T: A \ra T \otimes_{1,2} T^\vee] &&=
	\left[
	\begin{array}{cccccc}
		\mcS(l,l) &\ra &\bigoplus_{i,j\in\mcI} &\mcS(l,ij) &\otimes_\opk &\mcS(ij,l)\\
		\id_l & \mapsto &
\sum_{i,j,\a} & \alpha & \otimes_\opk & \bar\alpha
	\end{array}
	\right].
\end{align*}
All in all, we get the braiding to be the map
\begin{equation}
\label{eq:TV_braiding}
\mcS(l, X(Yr)) \ra \mcS(l, Y(Xr)), \quad
\pic[1.5]{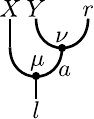} \mapsto \sum_{i,j,\a}\pic[1.5]{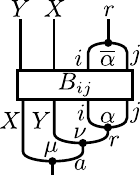}.
\end{equation}
For $M = F(X,\g)$, $N = F(Y,\d)$, the $T$-crossings are as given in Corollary~\ref{cor:ZS_equiv_D_linear}.
Using these expressions, the braiding \eqref{eq:TV_braiding} and the monoidal structure \eqref{eq:TV_monoidal_struct}, one concludes that the left hand side of \eqref{eq:braided_funct_cond} is a family of linear maps $\mcS(l,X(Yr)) \ra \mcS(l, (YX)r)$, $i,j,l\in\mcI$, obtained from post-composition with morphisms $X(Yr) \ra (YX)r$, which in graphical calculus are as shown in Figure~\ref{fig:big-diag-calc}.
 
In the last equality there we used
\begin{equation}
\sum_{i,j} |i|\cdot |j|\cdot N_{ij}^r
= \sum_{i,j} |i|\cdot |j^*| \cdot N_{ir^*}^{j^*}
= \sum_{i} |i| \cdot |i| \cdot |r|
= \dim\mcS \cdot |r|.
\end{equation}

\begin{figure}[p]
	\vspace*{-2em}	
	\begin{align*}
		&\frac{1}{\dim\mcS}\sum_{i,j,\a}\pic[1.5]{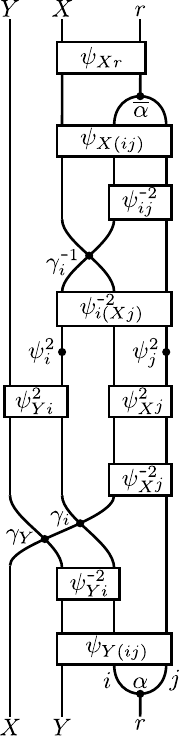}
		= \frac{1}{\dim\mcS}\sum_{i,j,\a} \pic[1.5]{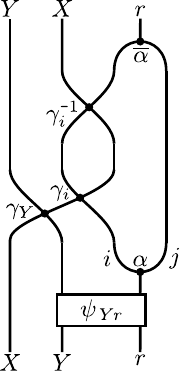} \cdot \frac{|i|\cdot |j|}{|r|}\\ \nonumber
		&= \frac{1}{\dim\mcS}\sum_{i,j,\a} \pic[1.5]{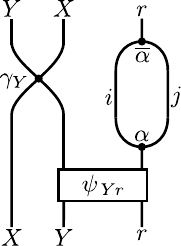} \cdot \frac{|i|\cdot |j|}{|r|}
		= \frac{1}{\dim\mcS}\sum_{i,j} \pic[1.5]{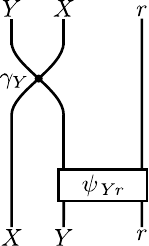} \cdot \frac{|i| \cdot |j| \cdot N_{ij}^r}{|r|}\\ \nonumber
		&= \pic[1.5]{425_braiding_4.pdf}.
	\end{align*}
	\caption{Left hand side of \eqref{eq:braided_funct_cond}.
		Here, in the first equality one uses the natural transformation property of $\psi$ and the half-braidings, in the third we abbreviate $N_{ij}^r = \dim \mcS(r,ij)$.}
	\label{fig:big-diag-calc}.
\end{figure}

\medskip

Substituting the braiding of $\mcZ(\mcS)$ as defined in \eqref{eq:ZS_br_tw}, one immediately finds the right hand side of \eqref{eq:braided_funct_cond} to be given by post-composition with the morphism in the last diagram of Figure~\ref{fig:big-diag-calc}.
The condition \eqref{eq:braided_funct_cond} then holds and hence $F$ is a braided equivalence.

Finally, recall that the twist of $M\in\mcC_{\opA^\mcS}$ is given by the morphism in \eqref{eq:twist}.
Using the calculation in Figure \ref{fig:big-diag-calc} and the expressions \eqref{eq:proper_ev-coevt} for (co-)evaluation maps one computes that the twist $\theta_{F(X,\g)}$ is a family of maps $\mcS(l,Xk) \ra \mcS(l,Xk)$, obtained from post-composition with
\begin{equation}
\pic[1.5]{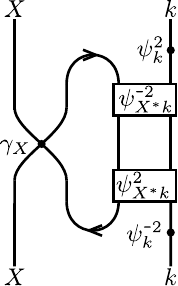} =
\pic[1.5]{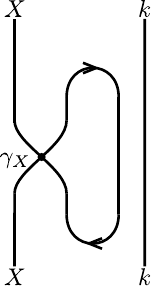} \stackrel{\eqref{eq:ZS_br_tw}}{=}
\pic[1.5]{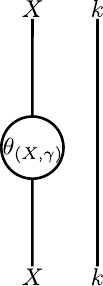},
\end{equation}
which is the same morphism as $F(\theta_{(X,g)})$.
$F$ is therefore an equivalence of ribbon fusion categories and with that the proof of Theorem \ref{thm:D_equiv_centre} is complete.

\newpage

\appendix
\section{Appendix}
\subsection{Useful identities for orbifold data and Wilson lines}
\label{app:identities}
Since $\a:T\otimes_2 T \ra T\otimes_1 T$, $\abar:T \otimes_1 T \ra T \otimes_2 T$ are $A$-$AAA$-bimodule morphisms, one has
\begin{equation}
\label{eq:a_proj_ids}
\pic[1.5]{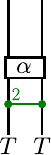}    = \pic[1.5]{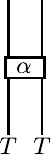}    = \pic[1.5]{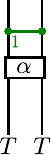}, \qquad
\pic[1.5]{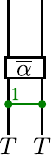} = \pic[1.5]{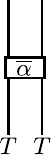} = \pic[1.5]{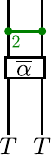},
\end{equation}
\begin{equation}
\label{eq:a_action_ids}
\pic[1.5]{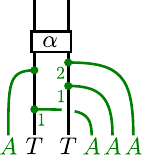}    = \pic[1.5]{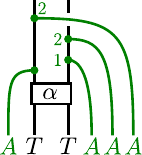}, \qquad
\pic[1.5]{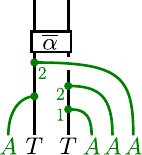} = \pic[1.5]{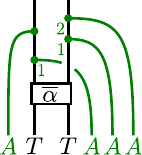}.
\end{equation}
In particular, one can commute the $\psi$-insertions with $\a$, $\abar$ as follows:
\begin{equation}
\label{eq:a_psis_proj}
\pic[1.5]{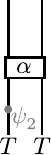} =
\pic[1.5]{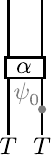} =
\pic[1.5]{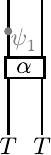} =
\pic[1.5]{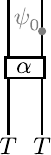}
\end{equation}
\begin{equation}
\label{eq:abar_psis_proj}
\pic[1.5]{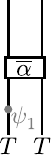} =
\pic[1.5]{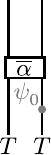} =
\pic[1.5]{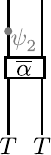} =
\pic[1.5]{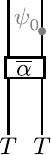}
\end{equation}
\begin{equation}
\label{eq:a_psis_action}
\pic[1.5]{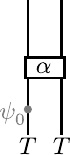} = \pic[1.5]{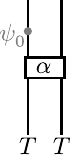}, \quad
\pic[1.5]{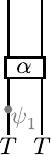} = \pic[1.5]{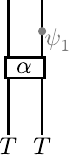}, \quad
\pic[1.5]{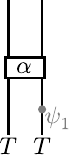} = \pic[1.5]{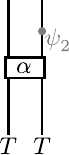}, \quad
\pic[1.5]{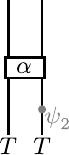} = \pic[1.5]{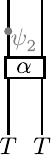},
\end{equation}
\begin{equation}
\label{eq:abar_psis_action}
\pic[1.5]{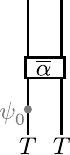} = \pic[1.5]{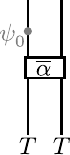}, \quad
\pic[1.5]{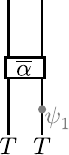} = \pic[1.5]{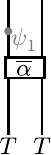}, \quad
\pic[1.5]{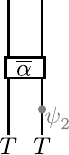} = \pic[1.5]{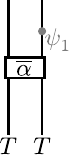}, \quad
\pic[1.5]{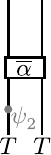} = \pic[1.5]{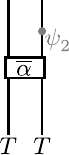},
\end{equation}
Similarly, $T$-crossings, being $A$-$AA$-bimodule morphisms $\tau_i:M \otimes_0 T \ra T \otimes_i M$, $\overline{\tau_i}:T \otimes_i M \ra M \otimes_0 T$, $i=1,2$, commute with $\psi$-insertions in the following way:
\begin{equation}
\label{eq:tau_psis}
\pic[1.5]{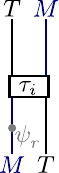} = \pic[1.5]{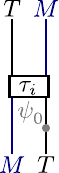}, \quad
\pic[1.5]{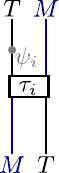} = \pic[1.5]{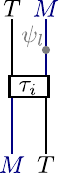}, \quad
\pic[1.5]{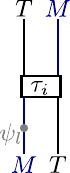} = \pic[1.5]{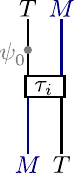}, \quad
\pic[1.5]{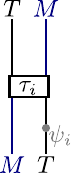} = \pic[1.5]{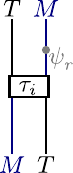},
\end{equation}
\begin{equation}
\label{eq:taubar_psis}
\pic[1.5]{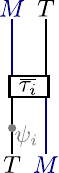} = \pic[1.5]{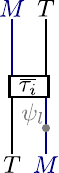}, \quad
\pic[1.5]{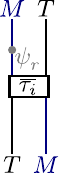} = \pic[1.5]{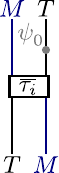}, \quad
\pic[1.5]{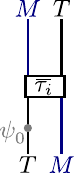} = \pic[1.5]{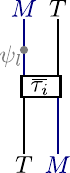}, \quad
\pic[1.5]{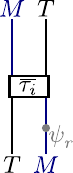} = \pic[1.5]{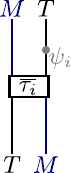},
\end{equation}

The following identities are dual versions of \eqrefO{6} and \eqrefO{7}:
\begin{equation*}
\begin{minipage}{0.4\textwidth}
\begin{equation}
\label{eq:O6-star} \tag{$\text{O9}'$}
\pic[1.5]{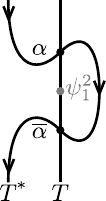} = \pic[1.5]{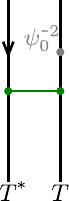}
\end{equation}
\end{minipage}\qquad
\begin{minipage}{0.4\textwidth}
\begin{equation}
\label{eq:O7-star} \tag{$\text{O10}'$}
\pic[1.5]{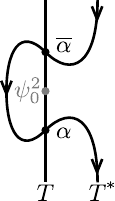} = \pic[1.5]{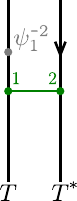}
\end{equation}
\end{minipage}
\end{equation*}

Finally, one can show that \eqrefT{1}-\eqrefT{3} imply:
\begin{gather*}
\begin{minipage}{0.33\textwidth}
\begin{equation}
\label{eq:T1p} \tag{$\text{T8}'$}
\pic[1.5]{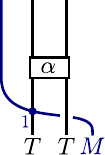} = \pic[1.5]{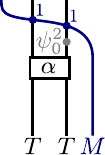}
\end{equation}
\end{minipage}
\begin{minipage}{0.33\textwidth}
\begin{equation}
\label{eq:T2p} \tag{$\text{T9}'$}
\pic[1.5]{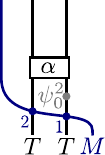} = \pic[1.5]{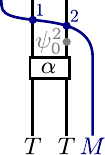}
\end{equation}
\end{minipage}
\begin{minipage}{0.33\textwidth}
\begin{equation}
\label{eq:T3p} \tag{$\text{T10}'$}
\pic[1.5]{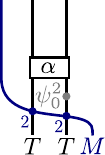} = \pic[1.5]{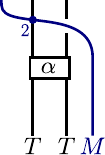}
\end{equation}
\end{minipage}\\
\begin{minipage}{0.33\textwidth}
\begin{equation}
\label{eq:T1b} \tag{$\text{T11}'$}
\pic[1.5]{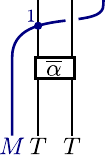} = \pic[1.5]{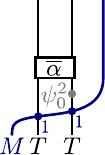}
\end{equation}
\end{minipage}
\begin{minipage}{0.33\textwidth}
\begin{equation}
\label{eq:T2b} \tag{$\text{T12}'$}
\pic[1.5]{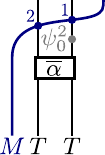} = \pic[1.5]{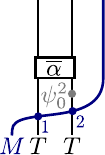}
\end{equation}
\end{minipage}
\begin{minipage}{0.33\textwidth}
\begin{equation}
\label{eq:T3b} \tag{$\text{T13}'$}
\pic[1.5]{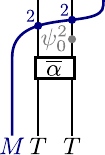} = \pic[1.5]{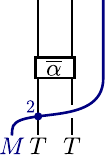}
\end{equation}
\end{minipage}\\
\begin{minipage}{0.33\textwidth}
\begin{equation}
\label{eq:T1bp} \tag{$\text{T14}'$}
\pic[1.5]{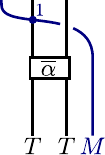} = \pic[1.5]{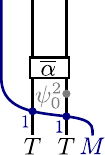}
\end{equation}
\end{minipage}
\begin{minipage}{0.33\textwidth}
\begin{equation}
\label{eq:T2bp} \tag{$\text{T15}'$}
\pic[1.5]{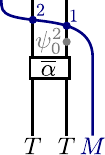} = \pic[1.5]{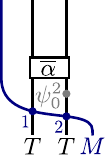}
\end{equation}
\end{minipage}
\begin{minipage}{0.33\textwidth}
\begin{equation}
\label{eq:T3bp} \tag{$\text{T16}'$}
\pic[1.5]{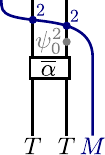} = \pic[1.5]{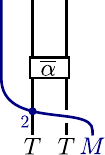}
\end{equation}
\end{minipage}
\end{gather*}
\subsection{Monadicity for separable biadjunctions}
\label{App:sep_biadj}
We will use string diagrams for $2$-categories, as reviewed in~\cite[Sec.\,6.1]{FSV}.

Let $\mcA$, $\mcB$ be categories, and let $X:\mcA \ra \mcB$, $Y: \mcB \ra \mcA$ be biadjoint functors with units and counits denoted by
\begin{equation}
\pic[1.5]{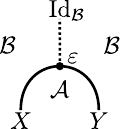},  \qquad
\pic[1.5]{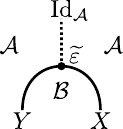}, \qquad
\pic[1.5]{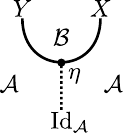},  \qquad
\pic[1.5]{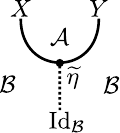}.
\end{equation}
Furthermore we will assume that this biadjunction is separable, i.e.\ the natural transformation
\begin{equation}
\pic[1.5]{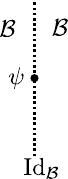} := \pic[1.5]{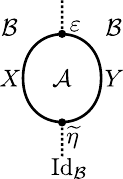}.
\end{equation}
is invertible.
The endofunctor $T:= [YX: \mcA \ra \mcA]$ becomes
a $\Delta$-separable
Frobenius algebra in the strict monoidal category $\End\mcA$ via the structure morphisms
\begin{align} \nonumber
&  \pic[1.5]{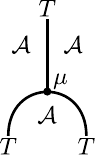}   := \pic[1.5]{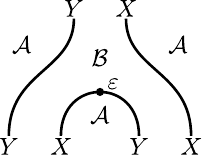},
&& \pic[1.5]{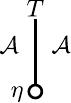}   := \pic[1.5]{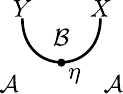},\\
&  \pic[1.5]{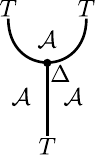} := \pic[1.5]{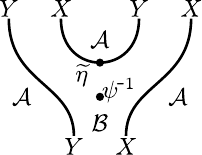},
&& \pic[1.5]{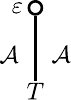} := \pic[1.5]{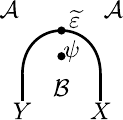}.
\end{align}
Let $\mcA^T$ be the category of \it{$T$-modules in $\mcA$}.
Its objects are pairs 
$$
	\big( \, U \in \mcA \,,\, [\rho:T(U) \ra U] \,\big)
$$ 
and a morphism $(U,\rho) \ra (U',\rho')$ is a morphism $[f:U \ra U'] \in \mcA$, such that the following diagrams commute:
\begin{equation}
\label{eq:T-mod_conds}
\begin{tikzcd}
TT(U)\arrow[r, "\mu_U"]\arrow[d, "T(\rho)"]
& T(U)\arrow[d, "\rho"] \\
T(U)\arrow[r, "\rho"]
& T
\end{tikzcd},\quad
\begin{tikzcd}
U \arrow[rd, "\id", swap] \arrow[r, "\eta_U"]
&T(U) \arrow[d, "\rho"]\\
&U
\end{tikzcd}, \quad
\begin{tikzcd}
T(U)\arrow[r, "T(f)"]\arrow[d, "\rho"]
& T(U')\arrow[d, "\rho'"] \\
U\arrow[r, "f"]
& U'
\end{tikzcd}.
\end{equation}

Let $\star$ be the category with only one object and only the identity morphism.
In what follows, it is going to be useful to identify any category $\mcA$ with the category of functors $\star \ra \mcA$ and natural transformations in the obvious way.
The conditions \eqref{eq:T-mod_conds} can then be written graphically as
\begin{align} \nonumber
\pic[1.5]{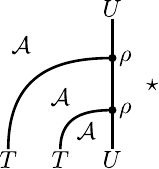}  = \pic[1.5]{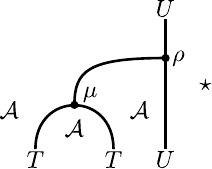}&, \quad
\pic[1.5]{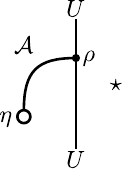}  = \pic[1.5]{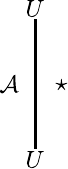},\\
\pic[1.5]{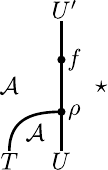} &= \pic[1.5]{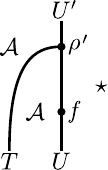}.
\end{align}

Define the functor $\widehat{Y}: \mcB \ra \mcA^T$ to be the same as $Y$, except that the image is equipped with the following $T$-action:
\begin{equation}
\pic[1.5]{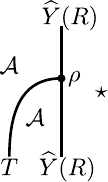} \quad := \pic[1.5]{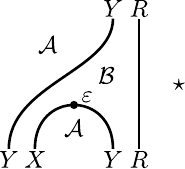}.
\end{equation}
\begin{defn}
Let $\mcA$ be a category.
\begin{itemize}
\item
An idempotent $[p:U \ra U] \in \mcA$ is called split, if it has a retract, i.e.\ a triplet $(S, e, r)$ where $S\in\mcA$, $e: S \ra U$, $r: U \ra S$, such that $e$ is mono, $r \circ e = \id_S$, $e \circ r = p$.
\item
$\mcA$ is called idempotent complete if every idempotent is split.
\end{itemize}
\end{defn}
\begin{prp}
\label{thm:idemp_compl_equiv}
If $\mcB$ is idempotent complete, then
$\widehat{Y}$ is an equivalence.
\end{prp}
\begin{proof}
We will give an inverse $\widehat{X}: \mcA^T \ra \mcB$.
Let $M \in \mcA^T$.
Define the following morphism $[p_M:X(M) \ra X(M)] \in \mcB$:
\begin{equation}
p_M := \pic[1.5]{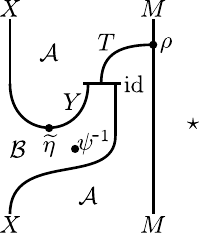}.
\end{equation}
One quickly checks that it is an idempotent.
Set $\widehat{X}(M) = \im p$.
To prove that it is indeed an inverse, one computes
\begin{equation}
\widehat{X}\widehat{Y}(R) = \im \pic[1.5]{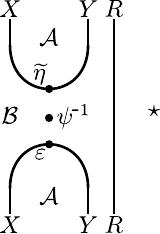}, \quad
\widehat{Y}\widehat{X}(M) = \im \pic[1.5]{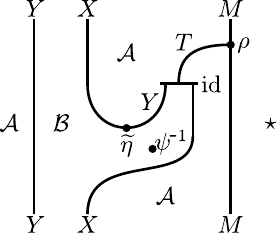}.
\end{equation}
The morphisms in pairs
\begin{equation*}
\pic[1.5]{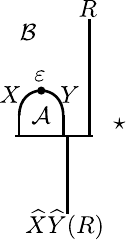}, \quad \pic[1.5]{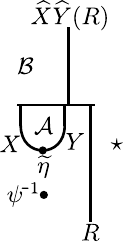} \qquad\text{and}\qquad
\pic[1.5]{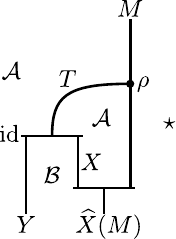}, \quad \pic[1.5]{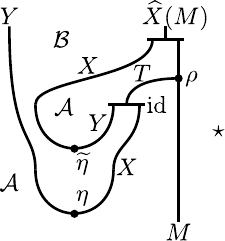}
\end{equation*}
are then inverses of each other.
\end{proof}

Now let $\mcA$, $X$, $Y$ (and hence also $\widehat{X}$, $\widehat{Y}$) in addition be $\opk$-linear additive categories and functors.
\begin{prp}
\label{thm:ssi_cond}
Suppose $\mcA$ is idempotent complete and finitely semisimple.${}^{\ref{fn:finitely-ssi}}$
Then so is $\mcA^T$.
\end{prp}
\begin{proof}
We first show
idempotent completeness of $\mcA^T$. Given an idempotent $p : M \to M$ in $\mcA^T$ and a retract $e : S \to M$, $r  : M \to S$ in $\mcA$ with $p = e \circ r$, one can equip $S$ with a $T$-action as follows,
\begin{equation}
\rho^S = \big[ T(S) \xrightarrow{T(e)} T(M) \xrightarrow{\rho^M} M \xrightarrow{r} S \big] \ .
\end{equation}
With respect to this action, $e$ and $r$ are morphisms in $\mcA^T$, so that $(S,e,r)$ becomes a retract in $\mcA^T$.

Next we show semisimplicity  of $\mcA^T$.
Let $M,N\in\mcA^T$ and let $\imath:M \ra N$ be mono in $\mcA^T$.
Since $\mcA$ is semisimple, there is $\widetilde{\pi}: N \ra M$ in $\mcA$, such that $\widetilde{\pi}\circ\iota = \id_M$.
Define
\begin{equation}
\pi := \pic[1.5]{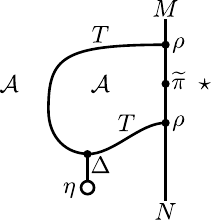}.
\end{equation}
One checks that $\pi:N\ra M$ is a morphism in $\mcA^T$ and
\begin{equation*}
\pi \circ \imath
= \pic[1.5]{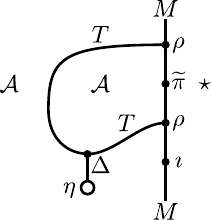} = \pic[1.5]{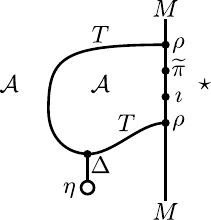} = \pic[1.5]{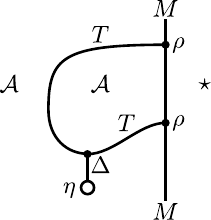} = \id_M.
\end{equation*}
It follows that $N \simeq M \oplus X$ for $X = \ker (\imath \circ \pi)$, 
and so all subobjects are direct summands. The kernel exists as it is the image of the idempotent $\id_N - \imath \circ \pi$.

For finiteness we show that every $T$-module $M\in\mcA^T$ is a submodule of an \it{induced} $T$-module, i.e.\ a one of the form
$\Ind(U) := [\star \xra{U} \mcA \xra{T} \mcA]$ for some object $U\in\mcA$.
Indeed, pick $U=M$ and the following morphisms:
$\Upsilon: M \ra \Ind(M)$ and $\Pi: \Ind(M) \ra M$
\begin{equation}
\Upsilon := \pic[1.5]{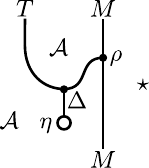}, \qquad
\Pi := \pic[1.5]{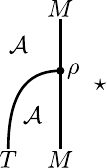}
\end{equation}
One can check that $\Upsilon$ and $\Pi$ are module morphisms and that $\Pi\circ\Upsilon = \id_M$, hence $M$ is indeed a submodule of $\Ind(M)$.
Every simple $T$-module is then a submodule of $\Ind(V)$ where $V\in\mcA$ is simple and since there are finitely many of those, $\mcA^T$ must have finitely many simple objects.
\end{proof}

\newcommand{\arxiv}[2]{\href{http://arXiv.org/abs/#1}{#2}}
\newcommand{\doi}[2]{\href{http://doi.org/#1}{#2}}

\end{document}